\documentclass[a4paper,11pt]{amsart}

\usepackage{amssymb}
\usepackage{amsmath}
\usepackage{amsthm,color}

\allowdisplaybreaks

\usepackage{amsfonts}
\usepackage{amssymb}
\usepackage{amsmath}
\usepackage{amsthm}
\usepackage{bbm}
\usepackage{verbatim}
\usepackage{url}
\usepackage[latin1]{inputenc} 
\usepackage[T1]{fontenc}

\usepackage{amsfonts}
\usepackage{amsmath,amssymb,amsthm}

\usepackage{mathrsfs}
\usepackage{verbatim}

\usepackage{graphicx}
\usepackage{ifpdf}

\newtheorem{thm}{Theorem}[section]

\newtheorem{lem}[thm]{Lemma}
\newtheorem{cor}[thm]{Corollary}
\newtheorem{lemma}[thm]{Lemma}

\theoremstyle{definition} 

\newtheorem{defn}[thm]{Definition}

\newcommand{\pa}{\varphi^t}

\newcommand{\R}{\mathbb{R}}

\newcommand{\Ht}{ \mathbb W_t^{[\lambda]} }
\newcommand{\Pt}{ \mathbb P_t^{[\lambda]} }
\newcommand{\Pto}{ \mathbb P_{t_1}^{[\lambda]} }
\newcommand{\Ptt}{ \mathbb P_{t_2}^{[\lambda]} }

\numberwithin{equation}{section}

\def\rr{{\mathbb R}}

\def\cm{{\mathcal M}}
\def\cn{{\mathcal N}}
\def\crz{{\mathcal R}}

\def\fz{\infty}
\def\az{\alpha}
\def\supp{{\mathop\mathrm{\,supp\,}}}

\def\lz{\lambda}

\def\ez{\epsilon}

\def\gz{{\gamma}}

\def\pa{\partial}

\def\gfz{\genfrac{}{}{0pt}{}}

\def\prz{{\partial}}
\def\gratx{{\nabla_{t,\,x}}}
\def\gratxo{{\nabla_{t_1,\,x_1}}}
\def\gratxt{{\nabla_{t_2,\,x_2}}}

\def\deltx{{\triangle_{t,\,x}}}

\def\prz{{\partial}}
\def\dmzo{{\,dx_1}}
\def\dmzt{{\,dx_2}}

\def\dmzdt{{\,dx_1\,dx_2\,dt_1\,dt_2}}
\def\supd{{\sup_{\gfz{t_1>0}{t_2>0}}}}

\def\inzf{{\int_0^\fz}}

\def\inrp{{\int_{\R_+}}}
\def\dinrp{{\iint_{\R_+\times \R_+}}}

\def\lozd{{L^1(\rlz)}}
\def\ltzd{{L^2(\rlz)}}
\def\lpzd{{L^p(\rlz)}}

\def\horiz{{H^1_{Riesz}(\rlz)}}

\def\homz{{H^1_{\cm}(\rlz)}}

\def\hsu{{H^1_{S_u}(\rlz)}}

\def\hnz{{H^1_{\cn_{h}}(\rlz)}}

\def\hrz{{H^1_{\crz_{h}}(\rlz)}}

\def\hrp{{H^1_{\crz_{P}}(\rlz)}}

\def\hnp{{H^1_{\cn_{P}}(\rlz)}}

\def\ls{\lesssim}
\def\gs{\gtrsim}

\def\ltz{{L^2(\rr_+)}}

\def\loz{{L^1(\rr_+)}}

\def\dint{\displaystyle\int}

\def\dlimsup{\displaystyle\limsup}
\def\dfrac{\displaystyle\frac}
\def\dsup{\displaystyle\sup}
\def\dlim{\displaystyle\lim}

\def\risz{{R_{S_\lz}}}
\def\riszo{{R_{S_\lz,\,1}}}
\def\riszt{{R_{S_\lz,\,2}}}
\def\sbz{{S_\lz}}
\def\hap{{H^1_{S_\lz}(\rlz)}}
\def\qlz{{\mathbb Q^{[\lz]}_t}}
\def\plz{{\mathbb P^{[\lz]}_t}}
\def\qlzo{{\mathbb  Q^{[\lz]}_{t_1}}}
\def\qlzt{{\mathbb Q^{[\lz]}_{t_2}}}
\def\plzo{{\mathbb P^{[\lz]}_{t_1}}}

\def\plzt{{\mathbb P^{[\lz]}_{t_2}}}

\def\r{\right}
\def\lf{\left}

\def\noz{\nonumber}

\def\rlz{\mathfrak R_+}

\def\XXint#1#2#3{{\setbox0=\hbox{$#1{#2#3}{\int}$}
     \vcenter{\hbox{$#2#3$}}\kern-.5\wd0}}

\allowdisplaybreaks
\begin{document}

\title[Product Hardy, BMO spaces associated with Bessel Schr\"odinger operator]{Product Hardy, BMO spaces and iterated commutators associated with Bessel Schr\"odinger operators }

\author{JORGE J. BETANCOR}
\address{Jorge J. Betancor, Departamento de Análisis Matemático, Universidad de La Laguna, Campus de Anchieta, Avda. Astrofísico Francisco Sánchez, s/n, 38271, La Laguna (Sta. Cruz de Tenerife), Spain.}
\email{jbetanco@ull.es}

\author{Xuan Thinh Duong}
\address{Xuan Thinh Duong, Department of Mathematics\\
         Macquarie University\\
         NSW 2019\\
         Australia
         }
\email{xuan.duong@mq.edu.au}

\author{Ji Li}
\address{Ji Li, Department of Mathematics\\
         Macquarie University\\
         NSW 2019\\
         Australia
         }
\email{ji.li@mq.edu.au}

\author{Brett D. Wick}
\address{Brett D. Wick, Department of Mathematics\\
         Washington University -- St. Louis\\
         St. Louis, MO 63130-4899 USA
         }
\email{wick@math.wustl.edu}

\author{Dongyong Yang$^\ast$}
\address{Dongyong Yang(Corresponding author), School of Mathematical Sciences\\
 Xiamen University\\
  Xiamen 361005,  China
  }
\email{dyyang@xmu.edu.cn }




\date{\today}



\thanks{$\ast$ Corresponding author}

\begin{abstract}
In this paper we establish the product Hardy spaces associated with the Bessel Schr\"odinger operator introduced by Muckenhoupt and Stein, and provide equivalent characterizations in terms of the Bessel Riesz transforms, non-tangential and radial maximal functions, and Littlewood--Paley theory, which are consistent with the classical product Hardy space theory developed by Chang and Fefferman. Moreover, in this specific setting, we also provide another characterization via the Telyakovski\'i transform, which further implies that the product Hardy space associated with this Bessel Schr\"odinger operator is isomorphic to the subspace of suitable ``odd functions'' in the standard Chang--Fefferman product Hardy space. Based on the characterizations of these product Hardy spaces, we study the boundedness of the iterated commutator of the Bessel Riesz transforms and functions in  the product BMO space associated with Bessel Schr\"odinger operator. We show that this iterated commutator is bounded above, but does not have a lower bound.
\end{abstract}

\maketitle




\section{Introduction and statement of main results}
\label{sec:introduction}
\setcounter{equation}{0}

There are several motivations for the research carried out in this paper.  Associated to the usual Laplacian $\Delta$ on $\mathbb{R}^n$ there are several important function spaces: the Hardy space $H^1(\mathbb{R}^n)$ and the space of functions with bounded mean oscillation BMO$(\mathbb{R}^n)$.  For the Hardy space, one has a family of equivalent norms that can be used to study the space: via maximal functions, square functions, area functions, Littlewood--Paley g-functions, Riesz transforms, and atomic decompositions \cite{S}.  Similarly the space BMO$(\mathbb{R}^n)$ has different ways that it can be studied: via commutators and via Riesz transforms \cite{crw}.  It has since become clear that the role of the differential operator greatly influences the harmonic analysis questions that one can consider.

The work of Betancor et al, \cite{bdt}, studied the Hardy space theory associated to a Bessel operator introduced by Muckenhoupt and Stein \cite{ms}, that serves as primary motivation for our paper.  Let $\lz\in \R_+:=(0, \fz)$ and
\begin{align}\label{defn of Besl Schr opr}
\sbz f(x):=-\frac{d^2}{dx^2}f(x)+\frac{\lz^2-\lz}{x^2}f(x),\,x>0.
\end{align}
The operator $\sbz$  in \eqref{defn of Besl Schr opr} is a positive self-adjoint operator on $\ltz$ and it can  be
written in divergence form as
$$\sbz=-x^{-\lz}Dx^{2\lz}Dx^{-\lz}=:A_\lz^\ast A_\lz,$$
where $A_\lz:=x^\lz Dx^{-\lz}$ and $A_\lz^\ast:= -x^{-\lz}D x^\lz$ is the adjoint operator of $A_\lz$. There has since been numerous investigations into harmonic analysis associated to this operator.  See for example \cite{bfbmt,bcfr,bcfr2,bfs,bdt,bhnv,dlwy,v08,yy}.

The main goal of this paper is to study the product theory of harmonic analysis associated to the operator $\sbz$.  In particular, we establish the product Hardy spaces associated with this Bessel Schr\"odinger operator and provide equivalent characterizations in terms of the Bessel Riesz transforms, non-tangential and radial maximal functions, and Littlewood--Paley theory, which are consistent with the classical product Hardy space theory developed by Chang and Fefferman \cite{CF1}.  We then also show that the commutators are bounded if the symbol belongs to a certain BMO space associated to the operator $\sbz$, but conversely this BMO does not characterize the boundedness of the commutator.     This last result is surprising since it is known in the classical multi-parameter setting that these iterated commutators in fact characterize the BMO of Chang and Fefferman (see \cite{fl,lppw}).  We now state our main results more carefully.

Throughout the paper, for every interval $I\subset \R_+$, we denote it by $I:=I(x,t):= (x-t,x+t)\cap \R_+$.
In the product setting $\mathbb{R}_+\times \mathbb{R}_+$, we define 
$ \rlz:= (\mathbb{R}_+\times \mathbb{R}_+, dx_1dx_2).$ We work with the domain $( \mathbb{R}_+\times \mathbb{R}_+)
\times ( \mathbb{R}_+\times \mathbb{R}_+)$ and its distinguished
boundary ${ \mathbb{R}_+\times  \mathbb{R}_+}$. For $x := (x_1,x_2)\in { \mathbb{R}_+\times
 \mathbb{R}_+}$, denote by $\Gamma(x)$ the product cone $\Gamma(x) :=
\Gamma_1(x_1)\times\Gamma_2(x_2)$, where $\Gamma_i(x_i) :=
\{(y_i,t_i)\in \mathbb{R}_+\times \mathbb{R}_+: |x_i-y_i| < t_i\}$
for $i := 1$, 2.

%
%
%
%
%

We now provide several definitions of
$H^1_{S_\lambda}$.  These spaces all end up being the same, which is one of the main results in this paper.   This requires some additional notation, but the careful reader will notice that the spaces are distinguished notationally by a subscript to remind how they are defined.


We first define the product Hardy spaces associated with the Bessel operator $ \sbz$ using the Littlewood--Paley area functions and square functions via the semigroups $\{T_t\}_{t>0}$, where $\{T_t\}_{t>0}$  can be the Poisson semigroup $\{e^{-t\sqrt{ \sbz}}\}_{t>0}$ or the heat semigroup $\{e^{-t \sbz}\}_{t>0}$. We note that the definition via heat semigroup was covered in  \cite{CDLWY} in a more general setting.

Given a function $f$ on $L^2(\rlz)$, the Littlewood--Paley {area
function}~$Sf(x)$, $x:=(x_1,x_2)\in \R_+\times\R_+$, associated with the operator $S_\lambda$ is
defined as
\begin{align}\label{esf}
     Sf(x)
    := \bigg(\iint_{\Gamma(x) }\Big|
        t_1\pa_{t_1}T_{t_1}\,  t_2\pa_{t_2}T_{t_2}  f(y_1,y_2)\Big|^2\
        {dy_1dy_2 dt_1  dt_2 \over t_1^2 t_2^2}\bigg)^{1\over2}.
   \end{align}

The square
function~$g(f)(x)$, $x:=(x_1,x_2)\in \R_+\times\R_+$, associated with the operator $S_\lambda$ is
defined as
\begin{align}\label{egf}
     g(f)(x)
    := \bigg(\int_{0}^\infty\int_{0}^\infty\Big|
        t_1\pa_{t_1}T_{t_1}\,  t_2\pa_{t_2}T_{t_2}  f(x_1,x_2)\Big|^2\
        { dt_1  dt_2 \over t_1 t_2 }\bigg)^{1\over2}.
   \end{align}





We define the product Hardy space via the Littlewood--Paley square functions as follows.   
\begin{defn}\label{def of Hardy space via g function}
The
    {Hardy space $H^1_{g}( \rlz )$} associated with $S_\lz$ is defined as the completion of
    \[
        \{f\in L^2(\rlz) :
        \|g(f)\|_{L^1(\rlz)} < \infty\}
    \]
    with respect to the norm
    $
        \|f\|_{H^{1}_{g}(\rlz ) }
        := \|g(f) \|_{L^1( \rlz)},
    $
    where $g(f)$ is defined by \eqref{egf} with $T_t:=e^{-t\sqrt{ \sbz}}$ or $T_t:=e^{-t \sbz}$ .

\end{defn}

We now define the product Hardy space  via the Littlewood--Paley area functions as follows.   
\begin{defn}\label{def of Hardy space via S function}
The
    {Hardy space $H^1_{S}( \rlz )$} associated with $S_\lz$ is defined as the completion of
    \[
        \{f\in L^2(\rlz) :
        \|Sf\|_{L^1(\rlz)} < \infty\}
    \]
    with respect to the norm
    $
        \|f\|_{H^{1}_{S}(\rlz ) }
        := \|Sf \|_{L^1( \rlz)},
    $
    where $Sf$ is defined by \eqref{esf} with $T_t:=e^{-t\sqrt{ \sbz}}$ or $T_t:=e^{-t \sbz}$ .

\end{defn}

We now define another version of the Littlewood--Paley area function.
Let
$$\nabla_{t_1,\,y_1}:=(\partial_{t_1}, \partial_{y_1}),\,\, \nabla_{t_2,\,y_2}:=(\partial_{t_2}, \partial_{y_2}).$$
Then  the Littlewood--Paley area function $S_uf(x)$ for $f \in L^2(\rlz)$,
$x:=(x_1,x_2)\in \R_+\times\R_+$  is defined as
\begin{align}\label{esfu}
     S_uf(x)
    := \bigg(\iint_{\Gamma(x) }\big| \nabla_{t_1,\,y_1} e^{-t_1\sqrt{ \sbz}}\nabla_{t_2,\,y_2}e^{-t_2\sqrt{ \sbz}}(f)(y_1,y_2) \big|^2 \
       \ dy_1\,dy_2 dt_1dt_2\bigg)^{1\over2}.
\end{align}
Then naturally we have the following definition of the product Hardy space via the Littlewood--Paley area function $S_uf$.
\begin{defn} \label{def of Hardy space via Su function}
The
    {Hardy space $\hsu $ } is defined as the completion of
    \[
        \{f\in L^2(\rlz ) :
        \|S_uf\|_{L^1(\rlz)} < \infty\}
    \]
    with respect to the norm
    $
        \|f\|_{H^{1}_{S_u}( \rlz ) }
        := \|S_uf \|_{L^1(\rlz )}.
    $
\end{defn}

Next we define the product non-tangential and radial maximal functions via the heat semigroup and Poisson semigroup  associated to $S_\lz$.
For all $\az\in (0, \fz)$, $p\in[1,\fz)$,
$f\in\lpzd$ and $x_1,x_2\in \R_+$, let
\begin{align*}
&\cn^\az_{h}f(x_1, x_2):=\!\!\sup_{\gfz{|y_1-x_1|<\az t_1}{|y_2-x_2|<\az t_2}}\!\lf|e^{-t_1 \sbz}e^{-t_2 \sbz} f(y_1, y_2)\r|,\\
&\cn^\az_{P}f(x_1, x_2):=\!\!\sup_{\gfz{|y_1-x_1|<\az t_1}{|y_2-x_2|<\az t_2}}\!\lf|e^{-t_1\sqrt{ \sbz}}e^{-t_2\sqrt{ \sbz}} f(y_1, y_2)\r|
\end{align*}
be the product non-tangential maximal functions with aperture $\alpha$ via the heat semigroup and Poisson semigroup  associated to $S_\lz$, respectively. Denote $\cn^1_{h}f$ by $\cn_{h}f$ and $\cn^1_{P}f$ by $\cn_{P}f$.  Moreover, let
\begin{align*}
&\crz_{h}f(x_1, x_2):=\sup_{t_1>0,\,t_2>0}\lf|e^{-t_1 \sbz}e^{-t_2 \sbz} f(x_1, x_2)\r|,\\
&\crz_{P}f(x_1, x_2):=\sup_{t_1>0,\,t_2>0}\lf|e^{-t_1\sqrt{ \sbz}}e^{-t_2\sqrt{ \sbz}} f(x_1,x_2)\r|
\end{align*}
be the product radial maximal functions via the heat semigroup and Poisson semigroup  associated to $S_\lz$, respectively.

\begin{defn}\label{defn:maximal funct}
The
    {Hardy space $\homz$} associated to the maximal function ${ \cm}f$ is defined as the completion of the set
    \[
        \{f\in L^2(\rlz ) :
        \|{ \cm}f\|_{L^1(\rlz)} < \infty\}
    \]
    with the norm
    $
        \|f\|_{\homz }
        := \|{ \cm}f \|_{L^1(\rlz )}.
    $
Here ${ \cm}f$ is one of the following maximal functions:
$\cn_{h}f$, $\cn_{P}f$, $\crz_{h}f$ and $\crz_{P}f$.
\end{defn}

Next we recall the definition of the Riesz transforms associated with $S_\lz$.
Define
\begin{align}
R_{S_\lz} f = A_\lz S_\lz^{-1/2}f.
\end{align}
Then we consider the definition of the product Hardy space via the Bessel Riesz transforms $ \riszo(f)$ and $\riszt(f)$ on the first and second variable, respectively. 

\begin{defn} \label{def of Hardy space via riesz}
 The {product Hardy space $\horiz$ } is defined as the completion of
\begin{eqnarray}\label{Hardy space riesz}
&&\lf\{f\in\lozd\cap L^2(\rlz): \,\,  \riszo f,\ \riszt f,\  \riszo\riszt f\in\lozd\r\}\noz
\end{eqnarray}
endowed with the norm
\begin{eqnarray*}
\|f\|_{\horiz}:=\|f\|_\lozd+\| \riszo f\|_\lozd +\|\riszt f\|_\lozd +\| \riszo\riszt f\|_\lozd.
\end{eqnarray*}
\end{defn}

The first main result of this paper is as follows.

\begin{thm}\label{thm main}
 Let $\lz\in(1, \fz)$.  The product Hardy spaces
$H^1_g(\rlz)$, $H^1_S(\rlz)$, $\hsu $, $\hnz,$
$\hrz$, $\hrp$,
 $\hnp$ and $\horiz$
coincide and have equivalent norms.
\end{thm}

Because we have a family of equivalent norms we now choose to use $H^{1}_{S_\lambda}( \rlz ) $ to denote the product Hardy space associated to $S_\lz$. Based on the atomic decomposition of $H^{1}_{S_\lambda}( \rlz ) $, we see that we can
identify $H^{1}_{S_\lambda}( \rlz ) $ as a closed subspace of $L^{1}( \rlz ) $.

Based on the characterization of product Hardy space $H^{1}_{S_\lambda}( \rlz ) $ via Bessel Riesz transforms and the duality of $H^{1}_{S_\lambda}( \rlz ) $ with ${\rm BMO}_{S_\lambda}( \rlz ) $, we directly have the second result as a corollary:  the decomposition of ${\rm BMO}_{S_\lambda}( \rlz ) $. For the definition of ${\rm BMO}_{S_\lambda}( \rlz ) $, we refer to Section 7. The proof of this result is similar to the classical setting.

\begin{cor}\label{thm BMO decomposition}
The following two statements are equivalent:

${\rm (i)}$ $\varphi \in {\rm BMO}_{S_\lambda}( \rlz ) $;

${\rm(ii)}$ There exist $g_i\in L^\infty( \rlz )$, $i=1,2,3,4$, such that
$$ \varphi= g_1 +  R_{S_\lambda,\,1}(g_2) + R_{S_\lambda,\,2}(g_3) + R_{S_\lambda,\,1}R_{S_\lambda,\,2}(g_4).$$
\end{cor}

The second main result of this paper is to understand the structure of the space $H^{1}_{S_\lambda}( \rlz ) $. Given a function
$f\in L^1(\rlz)$, we now introduce the ``product odd extension'' as follows
\begin{align}\label{odd}
f_o(x_1,x_2):=\left\{
\begin{aligned}
 & f(x_1,x_2), &\quad x_1>0, x_2>0; \\
 -& f(-x_1,x_2), &\quad x_1<0, x_2>0; \\
 & f(-x_1,-x_2), &\quad x_1<0, x_2<0;\\
 -& f(x_1,-x_2), &\quad x_1>0, x_2<0.
\end{aligned}
\right.
\end{align}
Note that for this odd extension, we have, for any fixed $x_2\in \R$,
$ f_o(x_1,x_2) = -f_o(-x_1,x_2) $; and for any fixed $x_1\in\R$,
$ f_o(x_1,x_2) = -f_o(x_1,-x_2) $.

Then we define the Hardy space $H^1_o(\rlz)$ as follows:
$$H^1_o(\rlz):=\{f\in L^1(\rlz): f_o\in H^1(\R\times\R)\}$$
with the norm $\|f\|_{H^1_o(\rlz)}=\|f_o\|_{H^1(\R\times\R)}$,
where $H^1(\R\times\R)$ is the standard Chang--Fefferman product Hardy space (see \cite{CF1}).  This leads to the second main theorem:
\begin{thm}\label{thm main 2}
$H^1_{S_\lz}(\rlz)$ and $H^1_o(\rlz)$ coincide and they have equivalent norms.
\end{thm}

As a corollary and application of our  main Theorems \ref{thm main} and \ref{thm main 2}, we also have the following results. The first one is the comparison of the classical standard Hardy space $H^1(\rlz)$ and our Hardy space $H^{1}_{S_\lambda}( \rlz ) $.
For the definition and properties of $H^1(\rlz)$, we consider $\rlz$ as a product space of homogeneous type and
we refer to \cite{HLL2}.
For the dual space of $H^1(\rlz)$, which is the classical standard product BMO space on $\rlz$, we also refer to
\cite{HLL2}.
\begin{thm}\label{thm H1 com}
The classical product Hardy space $H^1(\rlz)$ is a proper subspace of $H^{1}_{S_\lambda}( \rlz )$. As a consequence, we obtain that
 ${\rm BMO}_{S_\lambda}( \rlz ) $ is a proper subspace of the classical product BMO space ${\rm BMO}(\rlz)$.
\end{thm}

We also can provide the following result regarding the product BMO space ${\rm BMO}_{S_\lambda}( \rlz ) $ and the iterated commutator $[ [b, R_{S_\lambda,1}],R_{S_\lambda,2}]$.
\begin{thm}\label{thm commutator}
Let $b\in  {\rm BMO}_{S_\lambda}( \rlz ) $. Then we have
\begin{align*}
\|[ [b, R_{S_\lambda,1}],R_{S_\lambda,2}] \|_{L^2( \rlz)\to L^2(\rlz)} \ls \|b\|_{{\rm BMO}_{S_\lambda} ( \rlz)},
\end{align*}
where the implicit constant is independent of $b$. However, the lower bound is NOT true. In particular, there exists a locally integrable function $b_0 \not\in {\rm BMO}_{S_\lambda}( \rlz )$ such that
\begin{align*}
\|[ [b_0, R_{S_\lambda,1}],R_{S_\lambda,2}] \|_{L^2( \rlz)\to L^2(\rlz)}\leq C_{b_0} <\infty,
\end{align*}
where the constant $C_{b_0}$ is related to $b_0$.
\end{thm}
At this point we remark that this is a novel and surprising result.  In the classical multi-parameter setting, it was shown by Ferguson and Lacey, \cite{fl}, and Lacey and Terwilleger, \cite{lt}, that these iterated commutators characterize the product $BMO$ of Chang and Fefferman.  See also \cite{lppw, lppw2} for the case of Riesz transforms.  Whereas in this case, that natural $BMO$ is sufficient for the boundedness of these commutators, but is not necessary.


The outline of the paper is as follows.  In Section 2 we study the pointwise upper bound of the heat kernel $\Ht(x,y)$ and Poisson kernel $\Pt(x,y)$ of the operator $S_\lz$. We point out that
although the potential here can be negative,  $\Ht(x,y)$ still satisfies the standard Gaussian upper bound and $\Pt(x,y)$ satisfies the standard Poisson upper bound.

In Section 3, we prove that for $f\in L^1(\rlz)$,
\begin{align}\label{key1}
\|g(f)\|_{L^1(\rlz)} \approx \|S(f)\|_{L^1(\rlz)}
\end{align}
and the implicit constants are independent of $f$. Here the Littlewood--Paley $g$-function and area functions can be defined using both heat semigroups and Poisson semigroups of
$S_\lz$. The main strategy here is the atomic decomposition, especially  the direction that the $g$-function implies the atomic decomposition, where we apply the Moser type inequality of the Poisson semigroup.

In Section 4, we prove that for $f\in L^1(\rlz)$,
\begin{align}\label{key2}
\|f\|_{H^1_{S}(\rlz)} &\lesssim \|f\|_{H^1_{S_u}(\rlz)}  \lesssim \|f\|_{H^1_{\mathcal N_P}(\rlz)}\lesssim \|f\|_{H^1_{\mathcal R_P}(\rlz)} \lesssim \|f\|_{H^1_{\mathcal R_h}(\rlz)}
\lesssim \|f\|_{H^1_{\mathcal N_h}(\rlz)}\\
&\lesssim \|f\|_{H^1_{S}(\rlz)}\nonumber
\end{align}
and the implicit constants are independent of $f$.  The main approach here is the use of  Merryfield's Lemma, atomic decomposition, and the Moser type inequality of the Poisson semigroup.

In Section 5, we prove that for $f\in L^1(\rlz)$,
\begin{align}\label{key3}
 \|f\|_{H^1_{\mathcal R_P}(\rlz)} \lesssim \|f\|_{H^1_{Riesz}(\rlz)}
\lesssim \|f\|_{H^1_{S}(\rlz)},
\end{align}
and the implicit constants are independent of $f$.
These inequalities, together with the loop in \eqref{key2}, imply that our main result Theorem \ref{thm main} holds.
The main tools we use here are the Cauchy--Riemann type equations associated with $S_\lz$ and the conjugate harmonic function estimates.

For the proof of our second main result Theorem \ref{thm main 2} is as follows.
We introduce a product Hardy type space $H^1_{\mathcal T}(\rlz)$ via the Telyakovski\'i transform on $\R_+$, which is also called the local Hilbert transform (see \cite{am} and \cite{Fr}), defined by
\begin{align*}
\mathcal T_{\mathbb R_+}f(x) =p.v.\int_{x\over2}^{3x\over2} {f(t) \over x-t}dt.
\end{align*}
The product Hardy type space $H^1_{\mathcal T}(\rlz)$
is defined as the completion of
$$\{f\in L^1(\rlz)\cap L^2(\rlz):\ \mathcal T_{\mathbb R_+,1}f, \mathcal T_{\mathbb R_+,2}f, \mathcal T_{\mathbb R_+,1}\mathcal T_{\mathbb R_+,2}f \in L^1(\rlz)\}$$
with respect to the norm
$$ \|f\|_{H^1_{\mathcal T}(\rlz)}:=\|f\|_{L^1(\rlz)}+\|\mathcal T_{\mathbb R_+,1}f\|_{L^1(\rlz)}+\|\mathcal T_{\mathbb R_+,2}f\|_{L^1(\rlz)}+\|\mathcal T_{\mathbb R_+,1}\mathcal T_{\mathbb R_+,2}f\|_{L^1(\rlz)}, $$
where $\mathcal T_{\mathbb R_+,1}$ denotes the Telyakovski\'i transform on the first variable and $\mathcal T_{\mathbb R_+,2}$ the second.
Then, to prove Theorem \ref{thm main 2}, we demonstrate that
$$ \|f\|_{H^1_{Riesz}(\rlz)} \approx \|f\|_{H^1_{\mathcal T}(\rlz)} \approx \|f\|_{H^1_o(\rlz)}.$$

Finally, as applications of our  main Theorems \ref{thm main} and \ref{thm main 2}, in Section 6 we provide the proof of Theorems \ref{thm H1 com} and \ref{thm commutator}.

Throughout the whole paper,
we denote by $C$ {positive constant} which
is independent of the main parameters, but it may vary from line to
line.
If $f\le Cg$, we then write $f\ls g$ or $g\gs f$;
and if $f \ls g\ls f$, we  write $f\approx g.$

%
%
%
%
%


\section{Heat kernel and Poisson kernel estimates}

The heat semigroup $\{\Ht\}_{t>0}$ generated by $-S_\lz$ is defined by
$$ \Ht(f)(x) =\int_0^\infty \Ht(x,y)f(y)dy, $$
where
$$  \Ht(x,y)= {(xy)^{1\over2}\over 2t}\ I_{\lz-{1\over2}}\Big( {xy\over 2t} \Big)\ e^{-{x^2+y^2\over 4t}},\quad t,x,y\in(0,\infty),  $$
see \cite{bdt}.
Here $I_\nu$ represents the modified Bessel function of the first kind and order $\nu$ (see \cite{Le} for the properties of $I_\nu$).

\subsection{Upper bounds of the heat kernel and Poisson kernel}




\begin{thm}\label{thm Gaussian upper bound}
The kernel $\Ht(x,y)$ of the heat semigroup $\{\Ht\}_{t>0}$ satisfies the Gaussian estimate.  Namely, there are positive
constants $C$ and~$c$ such that for $t > 0$,
\begin{equation}\label{Ga}
    \big|\Ht(x,y)\big|
    \leq {C\over \sqrt t} e^{-|x-y|^2\over c t}.
    \tag{Ga}
\end{equation}


\end{thm}
\begin{proof}
%
%

First, we note that
\begin{align*}
  \Ht(x,y)&= {(xy)^{1\over2}\over 2t}\ I_{\lz-{1\over2}}\Big( {xy\over 2t} \Big)\ e^{-{x^2+y^2\over 4t}}\\
  & =  \bigg({xy\over 2t}\bigg)^{1\over2}\ I_{\lz-{1\over2}}\Big( {xy\over 2t} \Big)\ \ e^{-{xy\over 2t}}\ \cdot {1\over \sqrt{2t}}e^{-{(x-y)^2\over 4t}}.
\end{align*}

Then we claim that there exists a positive constant $C$ such that for all $x,y,t>0$,
\begin{align}\label{claim C}
\bigg({xy\over 2t}\bigg)^{1\over2}\ \left|I_{\lz-{1\over2}}\Big( {xy\over 2t} \Big)\right|\ \ e^{-{xy\over 2t}}\ \leq C.
\end{align}
This would then prove the Theorem that the heat semigroup $\{\Ht\}_{t>0}$ satisfies the Gaussian estimate \eqref{Ga}.  To see the claim \eqref{claim C}, we first recall that (see \cite{Le})
\begin{align}\label{Inu 1}
 C^{-1}x^\nu \leq I_\nu(x) \leq Cx^\nu,\quad  0<x\leq 1,
\end{align}
and
that $I_\nu(x)$ has the expansion
$$ I_\nu(x) \sim {e^x\over \sqrt{2\pi x}}\bigg\{ 1-{\mu-1\over 8x} +{(\mu-1)(\mu-9)\over 2! (8x)^2} - {(\mu-1)(\mu-9)(\mu-25) \over 3! (8x)^3} +\cdots \bigg\} $$ with $\mu=4\nu^2$,
which gives that
\begin{align}\label{Inu 2}
 I_\nu(x) = {e^x\over \sqrt{2\pi x}} +\Psi(x),\quad {\rm with}\quad |\Psi(x)| \leq C_\nu e^x x^{-3/2} {\rm\ for\ } x>{1\over4}.
\end{align}

We  now proceed by case analysis.

Case 1:  $ 0< {xy\over 2t} \leq 1$.  Then from \eqref{Inu 1} we have
\begin{align*}
I_{\lz-{1\over2}}\Big( {xy\over 2t} \Big) \approx \Big(  {xy\over 2t}  \Big)^{\lz-{1\over2}}.
\end{align*}
So we get
\begin{align*}
\bigg({xy\over 2t}\bigg)^{1\over2}\ I_{\lz-{1\over2}}\Big( {xy\over 2t} \Big)\ \ e^{-{xy\over 2t}}\   \approx \bigg({xy\over 2t}\bigg)^{1\over2}\ \Big(  {xy\over 2t}  \Big)^{\lz-{1\over2}}\ e^{-{xy\over 2t}}\leq \bigg({xy\over 2t}\bigg)^{\lz}\ e^{-{xy\over 2t}}\leq C.
\end{align*}

Case 2:  $  {xy\over 2t} > 1$.   By  \eqref{Inu 2} it follows 
\begin{align*}
I_{\lz-{1\over2}}\Big( {xy\over 2t} \Big) = {e^{{xy\over 2t}}\over \sqrt{2\pi {xy\over 2t}}} +\Psi\Big({xy\over 2t}\Big).
\end{align*}
So we can write
\begin{align*}
\bigg({xy\over 2t}\bigg)^{1\over2}\ \left|I_{\lz-{1\over2}}\Big( {xy\over 2t} \Big)\right|\ \ e^{-{xy\over 2t}}
&=\left|\bigg({xy\over 2t}\bigg)^{1\over2}\ {e^{{xy\over 2t}}\over \sqrt{2\pi {xy\over 2t}}} \ e^{-{xy\over 2t}} + \bigg({xy\over 2t}\bigg)^{1\over2}\ \Psi\Big({xy\over 2t}\Big) \ e^{-{xy\over 2t}}\right|\\
&\leq {1\over\sqrt {2\pi} } + C\bigg({xy\over 2t}\bigg)^{1\over2}\  e^{{xy\over 2t}} \Big({{xy\over 2t}}\Big)^{-{3\over2}}\ e^{-{xy\over 2t}}\leq C.
\end{align*}
\end{proof}

As a direct consequence, we obtain the following corollary.
\begin{cor}
The heat kernel $\Ht(x,y)$ of the heat semigroup $\{\Ht\}_{t>0}$ satisfies the Davies--Gaffney estimate.  Namely, there are positive
constants $C$ and~$c$ such that for all open subsets
$U_1,\,U_2\subset \R_+$ and all $t > 0$,
\begin{equation}\label{DG}
    |\langle \Ht f_1, f_2\rangle|
    \leq C\exp\Big(-\frac{{\rm dist}(U_1,U_2)^2}{c\,t}\Big)
        \|f_1\|_{L^2(\R_+)}\|f_2\|_{L^2(\R_+)}
    \tag{DG}
\end{equation}
for every $f_i\in L^2(\R_+)$ with $\mbox{supp}\,f_i\subset U_i$,
$i=1,2$. Here ${\rm dist}(U_1,U_2) := \inf_{x\in U_1, y\in U_2}
|x-y|$.
\end{cor}

We now consider the integral kernel $\Pt(x,y)$ associated with the  Poisson semigroup generated by
$-\sqrt{S_\lz}$.
\begin{cor}\label{cor Poisson}
There exists a positive constant $C$ such that
\begin{align}\label{Poisson}
|\Pt(x,y)| \leq C { t\over t^2+(x-y)^2 }\quad {\rm for\ all\ } x,y,t>0.
\end{align}
\end{cor}
\begin{proof}
By the principle of subordination, we have
\begin{align}\label{subo}
\Pt(x,y) = {1\over 2\sqrt\pi} \int_0^\infty e^{-{1\over4s}} s^{-{3\over2}}  \mathbb W_{t^2s}^{[\lambda]} (x,y) ds.
\end{align}
Since $\Ht(x,y)$ satisfies the Gaussian upper bound (Ga), it is direct that
\eqref{Poisson} holds.
\end{proof}

\subsection{Finite  propagation speed}

Let us
recall the finite  propagation speed for the wave equation and spectral multipliers (see \cite{DLY}) and adapted to the Bessel operator $S_\lz$.
%
Since the heat kernel $\Ht(x,y)$ satisfies the Gaussian bound ${\rm (Ga)}$,
it follows from \cite[Theorem 3]{CS}
 that there exists a finite,
positive constant $c_0$ with the property that the Schwartz
kernel $K_{\cos(t\sqrt{S_\lz})}$ of $\cos(t\sqrt{S_\lz})$ satisfies
\begin{eqnarray}\label{e3.12} \hspace{1cm}
{\rm supp}K_{\cos(t\sqrt{S_\lz})}\subseteq
\big\{(x,y)\in {\mathbb R}_+\times {\mathbb R}_+: |x-y|\leq c_0 t\big\};
\end{eqnarray}
see also \cite{Si2}. By the Fourier inversion
formula, whenever $F$ is an even, bounded, Borel function with its Fourier transform
$\hat{F}\in L^1(\mathbb{R})$, we can write $F(\sqrt{S_\lz})$ in terms of
$\cos(t\sqrt{S_\lz})$. More specifically,  we have
\begin{eqnarray}\label{XCP}
F(\sqrt{S_\lz})=(2\pi)^{-1}\int_{-\infty}^{\infty}{\hat F}(t)\cos(t\sqrt{S_\lz})\,dt,
\end{eqnarray}
which, when combined with (\ref{e3.12}), gives
\begin{eqnarray}\label{e3.13} \hspace{1cm}
K_{F(\sqrt{S_\lz})}(x,y)=(2\pi)^{-1}\int_{|t|\geq c_0^{-1}|x-y|}{\hat F}(t)
K_{\cos(t\sqrt{S_\lz})}(x,y)\,dt,\qquad {\rm for\ every\ }\,x,y\in{\mathbb R}_+.
\end{eqnarray}

The following two results
are useful for certain estimates later. We refer to  \cite[Lemma~3.5]{HLMMY} for the following lemmas.

\begin{lemma}\label{lemma finite speed} Let $\varphi\in C^{\infty}_c(\mathbb R)$ be
even, $\mbox{supp}\,\varphi \subset (-c_0^{-1}, c_0^{-1})$, where $c_0$ is
the constant in  \eqref{e3.12}. Let $\Phi$ denote the Fourier transform of
$\varphi$. Then for every $\kappa=0,1,2,\dots$, and for every $t>0$,
the kernel $K_{(t^2S_\lz)^{\kappa}\Phi(t\sqrt{S_\lz})}(x,y)$ of the operator
$(t^2S_\lz)^{\kappa}\Phi(t\sqrt{S_\lz})$ which was defined by the spectral theory, satisfies
\begin{eqnarray*}
{\rm supp}\ \! K_{(t^2S_\lz)^{\kappa}\Phi(t\sqrt{S_\lz})}(x,y) \subseteq
\Big\{(x,y)\in \mathbb{R}_+\times \mathbb{R}_+: |x-y|\leq t\Big\}.
\end{eqnarray*}
\end{lemma}

%
%

\begin{lemma}\label{lemma finite speed 2}
  Let
$\varphi\in C^{\infty}_c(\mathbb R)$ be an even function with
$\int_{\R}\varphi=2\pi$, ${\rm supp}\,\varphi \subset(-1,1)$.
  For every
$m=0,1, 2,\dots$, set $\Phi(\xi):={\hat \varphi}(\xi)$,
$
\Phi^{(m)}(\xi):={d^m\over d\xi^m} \Phi(\xi)
.
$
 Let $\kappa, m\in \mathbb{N}$ and $\kappa + m\in
2\mathbb{N}$. Then   for any $t>0$, the kernel
$K_{(t\sqrt{S_\lz})^{\kappa}\Phi^{(m)}(t\sqrt{S_\lz})}(x,y)$ of
  $(t\sqrt{S_\lz})^{\kappa}\Phi^{(m)}(t\sqrt{L})$ satisfies
\begin{eqnarray}\label{e2.5}\hspace{1.2cm}
{\rm supp}\ \! K_{(t\sqrt{S_\lz})^{\kappa}\Phi^{(m)}(t\sqrt{S_\lz})}
\subseteq \big\{(x,y)\in \mathbb{R}_+\times \mathbb{R}_+:\,|x-y|\leq t\big\}
\end{eqnarray}
 and
\begin{eqnarray}\label{ek}
\big|K_{(t\sqrt{S_\lz})^{\kappa}\Phi^{(m)}(t\sqrt{S_\lz})}(x,y)\big|
\leq C  \, t^{-1}
\end{eqnarray}
for any $x,y\in \mathbb{R}_+.$

\end{lemma}

\subsection{Moser type inequality}

As the Moser type inequality, established in \cite[p.\,454]{bcfr2}, we mean the following:  For any $x_0\in\R_+$, $t_0\in(\R_+)$ and $0<r <t_0$,
\begin{align}\label{Moser}
|u(t_0, x_0)|\ls \lf[\frac1{r^2}\int_{B((t_0, \,x_0), \, r)}|u(t, x)|^p\,dxdt\r]^{1/p},
\end{align}
where $u(t,x) = \Pt(f)(x)$ with $f\in L^1(\R_+)$ and $B((t_0, \,x_0), \, r):=\{(t,x)\in \R_+\times\R_+:\ |(t,x)-(t_0,x_0)|<r\}$. We mention that the inequality was proved in \cite{bcfr2} for $p=2$. However, an iteration argument shows that it also holds for
general $p>0$.

\subsection{Kernel estimates of Riesz transform $R_{S_\lz}$}\label{subsec: riesz}

We next note that the Riesz transform $R_{S_\lambda}$  related to Bessel operator $S_\lambda$
is bounded on $L^2(\mathbb{R}_+)$, and the kernel $R_{S_\lambda}(x,y)$ of $R_{S_\lambda}$
satisfies the following size and regularity properties, as proved in \cite[Proposition 4.1]{bfbmt}:

There exists $C>0$ such that for every $x,y\in \mathbb{R}_+$ with $x\not=y$,
\begin{eqnarray*}
&& (i)\  |R_{S_\lambda}(x,y)|\leq {C\over |x-y|};\\
&& (ii)\  \Big|{\partial\over \partial  x} R_{S_\lambda}(x,y)\Big| + \Big|{\partial\over \partial y} R_{S_\lambda}(x,y)\Big| \leq {C\over |x-y|^2}.\\
\end{eqnarray*}

We also recall the following version of upper bound for $R_{S_\lambda}(x,y)$, which will be very useful
in Section 6, connecting to the Telyakovski\'i transforms.  There exists constant $C>0$ such that
\begin{eqnarray*}
&& (i)'\  |R_{S_\lambda}(x,y)|\leq C{y^\lz\over x^{\lz+1}},\quad 0<y<{x\over2},\\
&& (ii)'\ |R_{S_\lambda}(x,y)|\leq C{x^{\lz+1}\over y^{\lz+2}},\quad y>{2x},\\
& & (iii)\ R_{S_\lambda}(x,y)= {1\over\pi}{1\over x-y} +O\bigg( {1\over x} \Big(1+\log_+{\sqrt{xy}\over |x-y|}\Big) \bigg),\quad {x\over2}<y<2x.
\end{eqnarray*}

\section{Proof of Equation \eqref{key1}}
\label{s2}


\subsection{Product Hardy spaces $H^1_{S}(\rlz)$ and atoms}

We first consider $H^1_{S}(\rlz)$ as in Definition \ref{def of Hardy space via S function} via $T_t= e^{-tS_\lz}$. Note that the kernel $\Ht(x,y)$ of $e^{-tS_\lz}$
satisfies the Gaussian upper bound \eqref{Ga}. Thus, $H^1_{S}(\rlz)$ falls into the scope of the Hardy space theory developed in \cite{DLY}.
We recall the definition and the atomic decomposition as follows.

First we recall the dyadic intervals in $\R_+$ as $\mathcal{D}(\R_+):=\cup_{n\in\mathbb Z} \mathcal{D}_n(\R_+)$, where
for each $n\in\mathbb Z$, $\mathcal{D}_n(\R_+):=\{ ({k\over 2^n}, {k+1\over 2^n}]:\ k\in\mathbb Z {\rm\ and\ } k\geq0\}$.
For $I,J\in \mathcal{D}(\R_+)$, we use $R:=I\times J$ to denote the dyadic rectangles in $\R_+\times\R_+$.
Then we denote all the dyadic rectangles in $\R_+\times\R_+$ by $\mathcal{D}(\R_+\times\R_+)=\cup_{n_1,n_2\in\mathbb Z} \mathcal{D}_{n_1,n_2}(\R_+)$, where
$\mathcal{D}_{n_1,n_2}(\R_+)=\{R=I\times J:\ I\in \mathcal{D}_{n_1}(\R_+), \ J\in \mathcal{D}_{n_2}(\R_+) \}$.

Suppose $\Omega\subset \R_+\times\R_+$
is an open set of finite measure. Denote by $m(\Omega)$ the maximal dyadic
subrectangles of $\Omega$. Let $m_1(\Omega)$  denote those dyadic
subrectangles $R\subseteq \Omega, R=I\times J$, that are maximal in
the $x_1$ direction. In other words if $S=I'\times J\supseteq R$ is
a dyadic subrectangle of $\Omega$, then $I=I'.$ Define $m_2(\Omega)$
similarly.   Let
${\widetilde \Omega}:=\big\{(x_1,x_2)\in \R_+\times\R_+:
\mathcal{M}_s(\chi_{\Omega})(x_1,x_2)>1/2\big\},
$
\noindent where ${\mathcal M}_s$ is the strong maximal operator on $\R_+\times \R_+$ defined as
$$
   \mathcal{M}_s(f)(x_1,x_2):=\sup_{R:\ {\rm \ rectangles\ in\ } \R_+\times \R_+,\ (x_1,x_2)\in R}{1\over |R|}\int_R|f(y)|dy.
$$

\noindent For any $R=I\times J\in m_1(\Omega)$, we  set
$\gamma_1(R)=\gamma_1(R, \Omega)=\sup{|l|\over |I|},$ where the
supremum is taken over all dyadic intervals $l: I\subset l$ so that
$l\times J\subset {\widetilde \Omega}$. Define $\gamma_2$ similarly.
Then Journ\'e's lemma, (in one of its forms, see for example \cite{J,P,HLLin}) says, for  any
$\delta>0$,
\begin{eqnarray*}
\sum_{R\in m_2(\Omega)} |R|\gamma_1^{-\delta}(R)\leq
c_{\delta}|\Omega|
 \ \ \ {\rm and}\ \ \
 \sum_{R\in m_1(\Omega)} |R|\gamma_2^{-\delta}(R)\leq c_{\delta}|\Omega|
 \end{eqnarray*}

\noindent for some  $c_{\delta}$ depending only on $\delta$, not on
$\Omega.$

We now recall the definition of a $(S_\lz, 2, M)$-atom.

\begin{defn}[\cite{DLY,CDLWY}]\label{def of product atom}  Let $M$ be a positive integer.
A function $a(x_1, x_2)\in L^2(\rlz)$ is called a
$(S_\lz, 2, M)$-atom     if it satisfies:

\medskip

\noindent $ 1)$  {\rm supp} $a\subset \Omega$, where $\Omega$
is an open set of $\R_+\times\R_+$ with finite measure;

\medskip

\noindent $ 2)$ $a$ can be further decomposed into
$
a=\sum\limits_{R\in m(\Omega)} a_R
$
\noindent where $m(\Omega)$ is the set of all maximal dyadic
subrectangles of $\Omega$, and there exist a series of functions $b_R $ belonging to the domain of
$S_\lz^{k_1}\otimes S_\lz^{k_2}$ in $L^2(\rlz)$, for each
$k_1, k_2=1, \cdots, M,$   such that

\smallskip

 (i) \ \ \ $a_R=\big(S_\lz^{M} \otimes S_\lz^{M}\big) b_R$;

 \smallskip

 (ii) \ \ {\rm supp}\  $\big(S_\lz^{k_1}  \otimes S_\lz^{k_2}\big)b_R\subset
10R$, \ \ $ k_1, k_2=0, 1, \cdots, M$;

\smallskip

 (iii) \ $\|a\|_{L^2(\rlz)}\leq
|\Omega|^{-{1\over 2}}$ and
$$
\sum_{R=I_R\times J_R\in m(\Omega)}\ell(I_R)^{-4M} \ell(J_R)^{-4M}
\Big\|\big(\ell(I_R)^2 \, S_\lz\big)^{k_1}\otimes \big(\ell(J_R)^2 \,
S_\lz\big)^{k_2} b_R\Big\|_{L^2(\rlz)}^2\leq |\Omega|^{-1}.
$$
\end{defn}

\medskip

We are now able to define an atomic Hardy space $H^1_{at,M}(\rlz)$ for $M>0$,
which is equivalent to the space
$H^1_{S}(\rlz)$.

\medskip

\begin{defn}[\cite{DLY}]\label{def-of-atomic-product-Hardy-space}
Let $M>0$. The Hardy spaces $H^1_{at,M}(\rlz)$
is defined as follows.
  We   say that $f=\sum\lambda_ja_j$ is an atomic
$(S_\lz, 2, M)$-representation of $f$ if $\{\lambda_j\}_{j=0}^\infty\in
\ell^1$, each $a_j$ is a $(S_\lz,  2, M)$-atom, and the sum converges in
$L^2(\rlz)$. Set
$$
\mathbb{H}^1_{at, M}(\rlz):=\big\{f\in L^2(\rlz): f {\rm\ has\
an\ atomic\ } (S_\lz, 2, M)-{\rm representation}\big\},
$$

\noindent with the norm given by
\begin{align}\label{Hp norm}
\|f\|_{\mathbb{H}^1_{at, M}(\rlz)}:=\inf \sum_{j=0}^\infty |\lambda_j|,
\end{align}
where the infimum is taken over sequences $\{\lambda_j\}_{j=0}^\infty$ such that $\sum_{j=0}^\infty |\lambda_j|<\infty$ and
$\sum_j\lambda_ja_j$
 is an atomic $(S_\lz, 2, M)$-representation of $f$.
 The space
$H^1_{at,M}(\rlz)$ is then
defined as the completion of
$\mathbb{H}^1_{at,M}(\rlz)$ with
respect to this norm.
\end{defn}

\begin{thm}[\cite{DLY}]\label{theorem of Hardy space atomic decom}
Suppose that $H^1_{S}(\rlz)$ is as in Definition \ref{def of Hardy space via S function} via $T_t= e^{-t S_\lz}$ and  that  $M\ge1$. Then
$
H^1_{S}(\rlz)=H^1_{at,M}(\rlz).
$
Moreover,
$
\|f\|_{H^1_{S}(\rlz)}\approx
\|f\|_{H^1_{at,M}(\rlz)},
$
where the implicit constants depend only on $M$.
\end{thm}

Second, we consider $H^1_{S}(\rlz)$ as in Definition \ref{def of Hardy space via S function} via $T_t= e^{-t\sqrt {S_\lz}}$. Note that the kernel $\Pt(x,y)$ of $e^{-t\sqrt{S_\lz}}$
satisfies the Poisson upper bound. In fact, following the same approach and techniques in the proof of \cite[Theorem 3.4]{DLY}, we also obtain the following result in terms of the Poisson semigroup.
\begin{thm}\label{theorem of Hardy space atomic decom Poisson}
Suppose that $H^1_{S}(\rlz)$ is as in Definition \ref{def of Hardy space via S function} via $T_t= e^{-t\sqrt {S_\lz}}$ and that  $M\ge1$. Then
$
H^1_{S}(\rlz)=H^1_{at,M}(\rlz).
$
Moreover,
$
\|f\|_{H^1_{S}(\rlz)}\approx
\|f\|_{H^1_{at,M}(\rlz)},
$
where the implicit constants depend only on $M$.
\end{thm}

Based on the atomic decomposition above, we now show that the Hardy spaces $
H^1_{g}(\rlz)$ and $H^1_{S}(\rlz)$ coincide and they have equivalent norms.
\begin{thm}\label{theorem of Hardy space g Hardy space S}
Suppose $H^1_{g}(\rlz)$ is as in Definition \ref{def of Hardy space via g function} via $T_t= e^{-t\sqrt {S_\lz}}$.
 Then
$
H^1_{g}(\rlz)=H^1_{S}(\rlz).
$
Moreover,
$
\|f\|_{H^1_{S}(\rlz)}\approx
\|f\|_{H^1_{g}(\rlz)}
$
where the implicit constants are independent of $f$. Similar result holds for the Hardy space $H^1_{g}(\rlz)$  as in Definition \ref{def of Hardy space via g function} via 
$T_t= e^{-t^2 S_\lz}$.

\end{thm}

\begin{proof}
Consider the Hardy space $H^1_{g}(\rlz)$ as in Definition \ref{def of Hardy space via g function} via $T_t= e^{-t\sqrt {S_\lz}}$.
We first show that
$$
H^1_{S}(\rlz)\subset H^1_{g}(\rlz).
$$
Let $f\in H^1_{S}(\rlz)$. According to Theorem \ref{theorem of Hardy space atomic decom Poisson},
$H^1_{S}(\rlz)= H^1_{at,M}$ for $M>0$. Then, $f\in H^1_{at,M}$, and to see that $f\in H^1_{g}(\rlz)$, it suffices to prove that
  for every  $(S_\lz, 2, M)$ atom $a$,
\begin{align}\label{ga}
\| g(   a)\|_{L^1(\rlz)}\ls1.
\end{align}
Following the same proof as in \cite[Lemma 3.6]{DLY}, we obtain that the above estimate holds for the Littlewood--Paley g-function defined via Poisson or heat semigroups as in \eqref{egf}.

Conversely,  we show that
$$
H^1_{g}(\rlz)\subset H^1_{S}(\rlz).
$$
To see this, we will show that we can derive atomic decomposition from the Littlewood--Paley g-function defined via Poisson or heat semigroups as in \eqref{egf}.

Let $f\in H^1_{g}(\rlz)\cap L^2(\rlz)$. We will first obtain the frame decomposition for $f$.  By using the reproducing formula, setting $\psi(x) := x^{M+1}\Phi(x)$ where $\Phi$ is defined as in Lemma \ref{lemma finite speed},  we can write
\begin{align}\label{e3 in section 2}
    \ \ \ \ f(x_1,x_2)
    &= \int_0^\infty\!\!\int_0^\infty
        \psi(t_1\sqrt{S_\lz})\psi(t_2\sqrt{S_\lz})(t_1\sqrt{S_\lz} e^{-t_1\sqrt{S_\lz}}\otimes t_2\sqrt{S_\lz} e^{-t_2\sqrt{S_\lz}})(f)(x_1,x_2){dt_1dt_2\over t_1t_2}\\
    &=\int_0^\infty\!\!\int_0^\infty\!\! \int_{\rlz}
        K_{\psi(t_1\sqrt{S_\lz})}(x_1,y_1)K_{\psi(t_2\sqrt{S_\lz})}(x_2,y_2)\nonumber\\
    &\hskip1cm(t_1\sqrt{S_\lz} e^{-t_1\sqrt{S_\lz}}\otimes t_2\sqrt{S_\lz} e^{-t_2\sqrt{S_\lz}})(f)(y_1,y_2)dy_1dy_2{dt_1dt_2\over t_1t_2}\nonumber\\
    &= \sum_{R\in \mathcal{D}(\R_+\times\R_+)}  \int_{T(R)} K_{\psi(t_1\sqrt{S_\lz})}(x_1,y_1)K_{\psi(t_2\sqrt{S_\lz})}(x_2,y_2)\nonumber\\
    &\hskip1cm(t_1\sqrt{S_\lz} e^{-t_1\sqrt{S_\lz}}\otimes t_2\sqrt{S_\lz} e^{-t_2\sqrt{S_\lz}})(f)(y_1,y_2)dy_1dy_2{dt_1dt_2\over t_1t_2}\nonumber\\
    &=: \sum_{R\in \mathcal{D}(\R_+\times\R_+)} s_R W_R, \nonumber
\end{align}
where
$$ s_R:=\sup_{(t_1,y_1,t_2,y_2)\in T(R)} |R|^{1/2} |(t_1\sqrt{S_\lz} e^{-t_1\sqrt{S_\lz}}\otimes t_2\sqrt{S_\lz} e^{-t_2\sqrt{S_\lz}})(f)(y_1,y_2)|$$
and when $s_R\not=0$,
\begin{align*}
 W_R&:={1\over s_R} \int_{T(R)} K_{\psi(t_1\sqrt{S_\lz})}(x_1,y_1)K_{\psi(t_2\sqrt{S_\lz})}(x_2,y_2)\\
 &\hskip2cm (t_1\sqrt{S_\lz}e^{-t_1\sqrt{S_\lz}}\otimes t_2\sqrt{S_\lz} e^{-t_2\sqrt{S_\lz}})(f)(y_1,y_2)dy_1dy_2{dt_1dt_2\over t_1t_2}.
\end{align*}
Here $T(R)= I\times [{|I|\over2},|I|) \times J\times [{|J|\over2},|J|)$,
$\{W_R\}_{R\in\mathcal{D}(\R_+\times\R_+)}$ is a family of frames, which satisfies the following conditions:

\smallskip
(1) we can further write
$ W_R=\sqrt{S_\lz}^{M}\otimes \sqrt{S_\lz}^{M} (w_R)$,
where
\begin{align*}
 w_R&:={1\over s_R} \int_{T(R)} K_{t_1^{M+1}\sqrt{S_\lz}\Phi(t_1\sqrt{S_\lz})}(x_1,y_1)K_{t_2^{M+1}\sqrt{S_\lz}\Phi(t_2\sqrt{S_\lz})}(x_2,y_2)\\
 &\hskip2cm (t_1\sqrt{S_\lz}e^{-t_1\sqrt{S_\lz}}\otimes t_2\sqrt{S_\lz} e^{-t_2\sqrt{S_\lz}})(f)(y_1,y_2)dy_1dy_2{dt_1dt_2\over t_1t_2};
\end{align*}

(2) ${\rm supp}\sqrt{S_\lz}^{k_1}\otimes \sqrt{S_\lz}^{k_2} (w_R)  \subset 3R, \quad k_1, k_2=0,1,\ldots,2M$\;
and

\smallskip
(3)
$  |(\ell(I)\sqrt{S_\lz})^{k_1}\otimes (\ell(J)\sqrt{S_\lz})^{k_2} (w_R)| \leq \ell(I)^{2M}\ell(J)^{2M}|R|^{-1/2},  \quad k_1, k_2=0,1,\ldots,2M.$

For the coefficients $\{s_R\}_R$, we claim that
\begin{align}\label{claim sR}
& \bigg\|\bigg(\sum_{k_1,k_2\in {\mathbb Z}}\Big[
\sum_{R\in \mathcal{D}_{k_1,k_2}(\R_+\times\R_+)}|R|^{-1/2}
|s_R|\chi_{R}(\cdot,\cdot)\Big]^2\bigg)^{1/2}\bigg\|_{L^1 ( \rlz )}\ls \|g(f)\|_{L^1 ( \rlz )}.
\end{align}

To see this, we first consider the estimate of $s_R$ for each $R\in \mathcal{D}_{k_1,k_2}(\R_+\times\R_+)$ with $k_1,k_2\in\mathbb{Z}$.
There exists $K,K_1,K_2\in\mathbb N$ such that for every $k_1,k_2\in \mathbb Z$ and $R\in \mathcal{D}_{k_1,k_2}(\R_+\times\R_+)$, we can find $(t_{j_1},x_{j_1},t_{j_2},x_{j_2})\in T(R)$, $0\leq j_1\leq K_1$, $0\leq j_2\leq K_2$, satisfying that
\begin{enumerate}
\item $T(R)\subset \cup_{j_1=0}^{K_1}\cup_{j_2=0}^{K_2} B_{1,j_1}\times B_{2,j_2}$, where for every $0\leq j_i\leq K_i$,
$B_{i,j_i} = B((x_{j_i}, t_{j_i}),2^{-k_i-2})$, $i=1,2$;

\item for every $i=1,2$ and $0\leq l\leq K_i$, card$\{j:\ 0\leq j\leq K_i {\rm \ and\ } B_{i,j}\cap B_{i,l}\not=\emptyset \}\leq K$.
\end{enumerate}
By geometric considerations we deduce that
%
 $\cup_{j_1=0}^{K_1}\cup_{j_2=0}^{K_2}B_{1,j_1}\times B_{2,j_2}\subseteq \{ (t_1,y_1,t_2,y_2): (y_1,y_2)\in 3R, \,
 2^{-k_1-2}< t_1< 3\cdot 2^{-k_1+1}, 2^{-k_2-2}< t_2< 3\cdot 2^{-k_2+1} \}.$   Hence,
\begin{eqnarray*}
&& |s_R \chi_{R}(x_1,x_2)|\,|R|^{-1/2}\\
&&=\sup_{(t_1,y_1,t_2,y_2)\in T(R)}  |(t_1\sqrt{S_\lz} e^{-t_1\sqrt{S_\lz}}\otimes t_2\sqrt{S_\lz} e^{-t_2\sqrt{S_\lz}})(f)(y_1,y_2)|
 \chi_{R}(x_1,x_2)\\
  &&\leq
\sum_{j_1=0}^{K_1}\sum_{j_2=0}^{K_2} \! \sup_{(t_1,y_1,t_2,y_2) \in B_{1,j_1}\times B_{2,j_2}} \!\! |(t_1\sqrt{S_\lz} e^{-t_1\sqrt{S_\lz}}\otimes t_2\sqrt{S_\lz} e^{-t_2\sqrt{S_\lz}})(f)(y_1,y_2)|
 \chi_{\overline{B_{1,j_1}}\times \overline{B_{2,j_2}}}(x_1,x_2),
\end{eqnarray*}
where for each $B_{i,j_i}$, we use $\overline{B_{i,j_i}}$ to denote the projection of $B_{i,j_i}$ onto $\R_+$.
Next, for any $q\in(0,1)$ and for $(x_1,x_2) \in \overline{B_{1,j_1}}\times \overline{B_{2,j_2}}$, 
\begin{eqnarray*}
&& \sup_{(t_1,y_1,t_2,y_2) \in B_{1,j_1}\times B_{2,j_2}}  |(t_1\sqrt{S_\lz} e^{-t_1\sqrt{S_\lz}}\otimes t_2\sqrt{S_\lz} e^{-t_2\sqrt{S_\lz}})(f)(y_1,y_2)|
 \\
 &\ls&
  \left( {1\over |B_{1,j_1}\times B_{2,j_2}|} \int_{B_{1,j_1}\times B_{2,j_2}}   |(s_1\sqrt{S_\lz} e^{-s_1\sqrt{S_\lz}}\otimes s_2\sqrt{S_\lz} e^{-s_2\sqrt{S_\lz}})(f)(y_1,y_2)|^q dy_1dy_2ds_1ds_2\right)^{1\over q} \\
  &\ls&
  \bigg[{1\over |B(x_{j_1},2^{-k_1})||B(x_{j_2},2^{-k_2})|} \int_{B(x_{j_1},2^{-k_1})\times B(x_{j_2},2^{-k_2})} \\
  &&\quad \left(  \int_{2^{-k_1-2}}^{2^{-k_1+3}}\int_{2^{-k_2-2}}^{2^{-k_2+3}}   |(s_1\sqrt{S_\lz} e^{-s_1\sqrt{S_\lz}}\otimes s_2\sqrt{S_\lz} e^{-s_2\sqrt{S_\lz}})(f)(y_1,y_2)|^q {ds_1ds_2\over s_1s_2}\right)\, dy_1dy_2 \bigg]^{1\over q} \\
   &\ls&
  \left[ \mathcal{M}_s\left(  \int_{2^{-k_1-2}}^{2^{-k_1+3}}\int_{2^{-k_2-2}}^{2^{-k_2+3}}   |(s_1\sqrt{S_\lz} e^{-s_1\sqrt{S_\lz}}\otimes s_2\sqrt{S_\lz} e^{-s_2\sqrt{S_\lz}})(f)(\cdot,\cdot)|^q {ds_1ds_2\over s_1s_2}\right)(x_1,x_2)   \right]^{1\over q},
\end{eqnarray*}
where the first inequality follows from the iteration of the Moser type inequality \eqref{Moser}.

    Therefore, by noting that $\sum_{R\in \mathcal{D}_{k_1,k_2}(\R_+\times\R_+)}\chi_{3R} \ls1$, we have
  \begin{align*}
&\sum_{R\in \mathcal{D}_{k_1,k_2}(\R_+\times\R_+)}
\sup_{(t_1,y_1,t_2,y_2)\in T(R)}  |(t_1\sqrt{S_\lz} e^{-t_1\sqrt{S_\lz}}\otimes t_2\sqrt{S_\lz} e^{-t_2\sqrt{S_\lz}})(f)(\cdot,\cdot)|
 \chi_{R}(x_1,x_2)\\
    &\ls \left[ \mathcal{M}_s\left(  \int_{2^{-k_1-1}}^{2^{-k_1}}\int_{2^{-k_2-1}}^{2^{-k_2}}   |(s_1\sqrt{S_\lz} e^{-s_1\sqrt{S_\lz}}\otimes s_2\sqrt{S_\lz} e^{-s_2\sqrt{S_\lz}})(f)(\cdot,\cdot)|^q {ds_1ds_2\over s_1s_2}\right)(x_1,x_2)   \right]^{1\over q}\\
   &=: F_{k_1,k_2}(x_1,x_2).
    \end{align*}

We take $q\in(0,1)$. Then,  the Fefferman--Stein vector valued maximal
inequality leads to
 \begin{align*}
& \bigg\|\bigg(\sum_{k_1,k_2\in {\mathbb Z}}\Big[
\sum_{R\in \mathcal{D}_{k_1,k_2}(\R_+\times\R_+)}|R|^{-1/2}
|s_R|\chi_{R}(\cdot,\cdot)\Big]^2\bigg)^{1/2}\bigg\|_{L^1 ( \rlz )}\\
&\ls    \bigg\|\bigg\{\sum_{k_1,k_2\in{\mathbb Z}}  F_{k_1,k_2}(x_1,x_2)^{2}
\bigg\}^{1/2} \bigg\|_{L^1 ( \rlz )}=     \bigg\|\bigg\{\sum_{k_1,k_2\in{\mathbb Z}}  F_{k_1,k_2}(x_1,x_2)^{2}
\bigg\}^{q/2} \bigg\|^{1\over q}_{L^{1\over q} ( \rlz )} \\
 &\ls\bigg\|\bigg\{\sum_{k_1,k_2\in{\mathbb Z}} \! \bigg[ \!   \int_{2^{-k_1-1}}^{2^{-k_1}}\!\!\int_{2^{-k_2-1}}^{2^{-k_2}}   |(s_1\sqrt{S_\lz} e^{-s_1\sqrt{S_\lz}}\otimes s_2\sqrt{S_\lz} e^{-s_2\sqrt{S_\lz}})(f)(x_1,x_2)|^q {ds_1ds_2\over s_1s_2}     \bigg]^{2\over q}
\bigg\}^{q\over2} \bigg\|^{1\over q}_{L^{1\over q} ( \rlz )} \\
 &\ls\bigg\|\bigg\{\sum_{k_1,k_2\in{\mathbb Z}}\!   \int_{2^{-k_1-1}}^{2^{-k_1}}\!\!\int_{2^{-k_2-1}}^{2^{-k_2}}    |(s_1\sqrt{S_\lz} e^{-s_1\sqrt{S_\lz}}\otimes s_2\sqrt{S_\lz} e^{-s_2\sqrt{S_\lz}})(f)(x_1,x_2)|^2 {ds_1ds_2\over s_1s_2}
\bigg\}^{1/2} \bigg\|_{L^{1} ( \rlz )} \\
 &\ls\bigg\|\bigg\{    \int_{0}^{\infty}\int_{0}^{\infty}    |(s_1\sqrt{S_\lz} e^{-s_1\sqrt{S_\lz}}\otimes s_2\sqrt{S_\lz} e^{-s_2\sqrt{S_\lz}})(f)(x_1,x_2)|^2 {ds_1ds_2\over s_1s_2}
\bigg\}^{1/2} \bigg\|_{L^1 ( \rlz )}\\
&= \|g(f)\|_{L^1 ( \rlz )},
\end{align*}
where $g(f)$ is the Littlewood--Paley g-function defined via the Poisson semigroup $T_t= e^{-t\sqrt{S_\lz}}$ as in \eqref{egf}. This shows that the claim \eqref{claim sR} holds.

Next, it suffices to show that from the frame decomposition as in \eqref{e3 in section 2},
\begin{align*}
    f= \sum_{R\in\mathcal{D}(\R_+\times\R_+)} s_R W_R
\end{align*}
with the condition \eqref{claim sR} for the coefficients $\{s_R\}_{R\in\mathcal{D}(\R_+\times\R_+)}$,
%
%
we can then derive the atomic decomposition.

To see this, we first denote
$$  \mathcal{S}(f)(x_1,x_2) := \bigg(\sum_{k_1,k_2\in {\mathbb Z}}\bigg[
\sum_{R\in \mathcal{D}_{k_1,k_2}(\R_+\times\R_+)}|R|^{-1/2}
|s_R|\chi_{R}(x_1,x_2)\bigg]^2\bigg)^{1/2}.$$
Then, we define
for each $\ell\in\mathbb{Z}$,
\begin{eqnarray*}
    \Omega_\ell&:=&\left\{(x_1,x_2)\in \R_+\times\R_+:  \mathcal{S}(f)(x_1,x_2)>2^\ell\right \},\\
    B_\ell&:=&\left\{R=I_{\alpha_1}^{k_1}\times I_{\alpha_2}^{k_2}:
        | R\cap \Omega_\ell|>{1\over 2}|R|,\  |R\cap \Omega_{\ell+1}|\leq {1\over 2}|R| \right\}, {\rm\ and}\\
    \widetilde{\Omega}_\ell&:=&\left\{(x_1,x_2)\in \R_+\times\R_+: \mathcal{M}_s(\chi_{\Omega_\ell})(x_1,x_2)>{1\over2} \right\},
\end{eqnarray*}
where $\mathcal{M}_s$ is the strong maximal function on $
\rlz$. Then we write
\begin{align*}
    f&= \sum_{R\in\mathcal{D}(\R_+\times\R_+)} s_R W_R= \sum_{\ell\in\mathbb{Z}} \sum_{R\in B_\ell} s_R W_R=\sum_{\ell\in\mathbb{Z}}\lambda_\ell a_\ell,
\end{align*}
where $ \lambda_\ell := 2^\ell |\widetilde{\Omega_\ell}| $
and
$$ a_\ell:= {1\over \lambda_\ell}\sum_{R\in B_\ell} s_R W_R.$$
Then we can verify that these $a_\ell$ are the $(S_\lz, 2, M)$ atoms as in Definition \ref{def of product atom},
i.e., we can obtain that
$ \|a_\ell\|_{L^2(\rlz)}\leq C |\widetilde{\Omega_\ell}|^{-{1\over2}}$, and that
$$ a_{\ell}=\sum_{R\in m(\widetilde{\Omega_\ell})}a_R,  $$
where these $a_R$'s satisfy the conditions as listed in Definition \ref{def of product atom}.

Then we have an atomic $(S_\lz, 2, M)$-representation of $f$. The details here follow from the proof of \cite[Theorem 4.1 (ii)]{LS} and \cite[Proposition 3.4]{CDLWY}.
\end{proof}

\section{Proof of inequalities \eqref{key2}}
\label{sec:second thm}

To this end, we will prove the chain of six inequalities as in \eqref{key2} by the following six steps, respectively.
In this section, we assume that $\lambda\geq 1$.

\medskip
{\bf Step 1:\ $\|f\|_{H^1_{S}( \rlz)} \leq
\|f\|_{H^1_{S_u}( \rlz)}$} for $f\in H^1_{S_u}( \rlz)\cap L^2(\rlz)$.

Note that from the definitions of the area functions $Sf$ and $S_uf$ in \eqref{esf} and \eqref{esfu} respectively, we have
for $f\in L^2(\rlz)$, $S(f)(x) \leq S_u(f)(x)$, which implies that $\|f\|_{H^1_{S}( \rlz)} \leq
\|f\|_{H^1_{S_u}( \rlz)}$.

\bigskip

{\bf Step 2:\ $\|f\|_{H^1_{S_u}( \rlz)} \ls
\|f\|_{H^1_{\cn_{P}}(\rlz)}$} for $f\in H^1_{\cn_{P}}( \rlz)\cap L^2(\rlz)$.



\medskip

We point out that the proof of $\|f\|_{H^1_{S_u}( \rlz)} \leq C
\|f\|_{H^1_{\cn_P}(\rlz)}$ is similar to the proof of Step 2 in \cite{dlwy2}, assuming that we
know a suitable version of the technical result originating from K. Merryfield \cite{KM}.
We now build up the right version of the Merryfield type lemma in this setting. See also a similar version of
Merryfield lemma in \cite{STY} for the Schr\"odinger operators on $\R^n$ with $n\geq3$.

Define the gradient as
\begin{align*}
\gratx u(t,x):=(\prz_t u, \prz_x u)
\end{align*}
and, as usual, the Laplace operator as
\begin{align*}
\deltx u(t,x):=\prz^2_t u+\prz_x^2 u.
\end{align*}



\begin{lem}\label{lem:Merryfield lem}
Let $\phi\in C^\fz_c(\R)$ be an even function, such that $\phi\geq0$, supp $\phi\subset (-1,1)$ and
$ \int_{-\infty}^\infty \phi(x)dx=1. $
Then
there exists a positive constant $C$ such that
for any $f\in\ltz$ with $u(t, x):=\plz f(x)$ satisfying $\sup\limits_{|y-x|<t}|u(t,y)|\in L^1(\R_+,dx)$, and for any
$g\in L^2(\R)$ with the condition that {\rm supp}\,$g\subset \R_+$,
\begin{align}\label{ee KM}
&\dinrp|\gratx u(t,x)|^2\lf|\phi_t\ast g(x)\r|^2 tdx\,dt\\
&\le C \lf[\inrp [f(x)]^2[g(x)]^2dx+\dinrp|u(t,x)|^2|Q_t(g)(x)|^2dx\,dt\r],\noz\nonumber
\end{align}
where $Q_t( g)(x) := \big( t \partial_t (\phi_t\ast g)(x),\ t \partial_x (\phi_t\ast g)(x),\psi_t\ast ( g)(x)\big)$ and $\psi(x) := x\phi(x)$, $\psi_t(x) := t^{-1}\phi({x\over t})$.
\end{lem}

\begin{proof}
The proof of this lemma  can be obtained by making minor modifications to the proof of \cite[Lemma 3.1]{KM} in the classical setting, i.e., when the Laplace operator replaces the Bessel operator. For the sake of completeness and for the reader's convenience we give a brief sketch of the proof.

Note that $$\deltx  (u^2) = 2|\gratx u|^2 +2 {\lambda^2-\lambda \over x^2} u^2. $$
We have
\begin{align*}
&2\dinrp|\gratx u(t,x)|^2\lf|\phi_t\ast g(x)\r|^2 tdx\,dt\\
&= \dinrp \deltx  (u^2)  \lf|\phi_t\ast g(x)\r|^2 tdx\,dt - 2 \dinrp {\lambda^2-\lambda \over x^2}  u^2(t,x)  \lf|\phi_t\ast g(x)\r|^2 tdx\,dt \\
&\leq   \dinrp \deltx  (u^2)  \lf|\phi_t\ast g(x)\r|^2 tdx\,dt
\end{align*}
since we have assumed the condition that $\lambda\ge 1$. Then integration by parts and following the proof of \cite[Lemma 3.1]{KM}
we get to the right-hand side of \eqref{ee KM}. For the construction of the function $\psi$, we refer to \cite[Equation (3.8)]{KM}.
\end{proof}

Next we have the following result for the product case. 
 Before stating our next Lemma, we introduce the notation $\phi_{t_1}\ast_{1} g(x_1, x_2)$, $\phi_{t_2}\ast_{2} g(x_1, x_2)$ and  $\phi_{t_1}\phi_{t_2}\ast_{1,\,2} g(x_1, x_2)$
to denote the convolution with respect to the first, second and both variables, respectively, where the function $\phi$ is the same as in Lemma \ref{lem:Merryfield lem}.

\begin{lem}\label{lem: product Merryfield lem}
Let $\phi$ be a smooth function as in Lemma \ref{lem:Merryfield lem}. There exists $C>0$ such that for every $f\in L^2(\rlz)$ and $g\in L^2(\mathbb R^2)$ with $\supp g\subset \mathbb R_+\times\mathbb R_+$, we have that
\begin{align*}
&\dinrp\dinrp\lf|\gratxo\gratxt u(t_1, t_2, x_1, x_2)\r|^2 \lf|\phi_{t_1}\phi_{t_2}\ast_{1,\,2} g(x_1, x_2)\r|^2 t_1 t_2\dmzdt\\
&\le C\bigg\{\dinrp[f(x_1, x_2)]^2 [g(x_1, x_2)]^2dx_1dx_2\\
&\quad+\inrp\dinrp\lf|\plzt f(x_1, x_2)\r|^2\lf[ Q^{(2)}_{t_2}(g)(x_1,x_2) \r]^2dx_2\,dt_2\dmzo\\
&\quad+\inrp\dinrp \lf|\plzo f(x_1, x_2)\r|^2 \lf[ Q^{(1)}_{t_1}(g)(x_1,x_2) \r]^2\,dx_1\,dt_1\dmzt\\
&\quad+\dinrp\dinrp | u(t_1,t_2,x_1, x_2)|^2  \Big| Q^{(1)}_{t_1} Q^{(2)}_{t_2}(g)(x_1,x_2) \Big|^2 {\dmzdt\over t_1t_2}\bigg\},
\end{align*}
where $u(t_1, t_2, x_1, x_2):=\plzo\plzt f(x_1, x_2)$, 
$Q^{(1)}_{t_1}(g)(x_1,x_2) := \big( t_1 \partial_{t_1} (\phi_{t_1}\ast_{1} g)(x_1,x_2),\ t_1 \partial_{x_1} (\phi_{t_1}\ast_{1} g)(x_1,x_2),  \psi_{t_1}\ast_1 ( g(\cdot,x_2) )(x_1)\big),  $
and $\psi(x_1) := x_1\phi(x_1)$, $\psi_{t_1}(x_1) := t_1^{-1}\psi({x_1\over t_1})$.
The definition of $Q^{(2)}_{t_2}(g)(x_1,x_2) $ is similar.
\end{lem}

This lemma follows from the iteration of Lemma \ref{lem:Merryfield lem}. We omit the details.

\begin{proof}[\bf Proof of $\|f\|_{H^1_{S_u}( \rlz)} \leq C
\|f\|_{H^1_{\cn_P}(\rlz)}$]

\hskip.2cm

For any $f\in \ltzd$ such that $\cn_P(f)\in \lozd$, and $\alpha>0$,
we define
\begin{eqnarray*}
{\mathcal A}(\alpha):=\lf\{(x,y)\in \R_+\times\R_+:\
\cm_s\big(\chi_{\{\cn_{P}(f)>\alpha\}}\big)(x,y)<\frac1{200}\r\},
\end{eqnarray*}
where $\cm_s$ is the strong maximal function on $\R_+\times \R_+$.
Our first objective is to see that
%
\begin{eqnarray}\label{step1}
&&\iint_{{\mathcal A}(\az)}S_u^2(f)(x_1,x_2)dx_1dx_2\leq\iiiint_{R^{*}}\big| t_1t_2\nabla_{t_1,\,y_1}\nabla_{t_2,\,y_2} u(t_1,t_2,y_1,y_2) \big|^2\frac{dy_1dy_2\,dt_1dt_2}{t_1t_2}, 
\end{eqnarray}
where for $t_1,\,t_2,\, y_1,\,y_2\in\R_+$, $R(y_1, y_2, t_1, t_2):=I(y_1,t_1)\times I(y_2,t_2)$ and
$$R^{*}:=\left\{(y_1,y_2,t_1,t_2)\in \R_+\times\R_+\times\R_+\times\R_+:\ \dfrac{|\{ \cn_P(f)>\alpha\}\cap
R(y_1,y_2,t_1,t_2)|} {\ |R(y_1,y_2,t_1,t_2)|}<\frac1{200}\right\}.$$

Let $ g(x,y):=\chi_{\{ \cn_{P}(f)\leq \alpha\}}(x,y)$ and $\phi\in C^\fz_c(\R)$ such that
$\supp(\phi)\subset (-1, 1)$, $\phi\equiv 1$ on $(-1/2,1/2)$ and $0\le \phi(x)\le 1$ for all $x\in\R$.
Then for
$(x_1,x_2,t_1,t_2)\in R^{*}$, we have
\begin{align}\label{step2}
 \phi_{t_1}\phi_{t_2}\ast g(x_1, x_2)&\ge \iint_{  \{ \cn_{P}(f)\leq
\alpha\} \cap R(x_1,\,x_2,\,t_1/2,\,t_2/2) }dy_1dy_2\gs1.
\end{align}
Combining (\ref{step1}) and (\ref{step2}), and using Lemma \ref{lem: product Merryfield lem}, we have
\begin{eqnarray*}
&&\iint_{{\mathcal A}(\az)}S_u^2(f)(x_1,x_2)dx_1dx_2\\
&&\quad \ls \iiiint_{R^{*}}\big| t_1t_2\nabla_{t_1,\,y_1}\nabla_{t_2,\,y_2} u(t_1,t_2,y_1,y_2) \big|^2
\lf| \phi_{t_1}\phi_{t_2}\ast g(y_1, y_2) \r|^2\frac{dy_1dy_2\,dt_1dt_2}{t_1t_2}\nonumber\\
&&\quad\ls\bigg\{\iint_{\R_+\times\R_+}[f(x_1, x_2)]^2 [g(x_1, x_2)]^2 dx_1dx_2\\
&&\quad\quad+\int_{\R_+}\iint_{\R_+\times\R_+}\lf| (\plzt f)(x_1, x_2)\r|^2\lf| Q^{(2)}_{t_2}(g)(x_1,x_2) \r|^2{dx_2\,dt_2 dx_1\over t_2}\\
&&\quad\quad+\int_{\R_+}\iint_{\R_+\times\R_+} \lf|(\plzo f)(x_1, x_2)\r|^2 \lf| Q^{(1)}_{t_1}(g)(x_1,x_2) \r|^2\,{dx_1\,dt_1 dx_2\over t_1}\\
&&\quad\quad+\dinrp\dinrp | u(t_1,t_2,x_1, x_2)|^2  \Big| Q^{(1)}_{t_1} Q^{(2)}_{t_2}(g)(x_1,x_2) \Big|^2 {\dmzdt\over t_1t_2}\bigg\}.
\end{eqnarray*}
By an argument analogous to that in \cite{dlwy2} (see also \cite{KM}), we see that
\begin{eqnarray*}
\iint_{\mathcal A(\alpha)}S_u^2(f)(x_1,x_2)\,dx_1dx_2
&&\ls \alpha^2 |\{(x_1,x_2)\in\R\times \R:\,\,\cn_{P}(f)(x_1,x_2)>\alpha\}|\noz\\
&&\quad+ \iint_{\{\cn_{P}(f)\leq
\alpha\}}|\cn_{P}(f)(x_1,\,x_2)|^2dx_1dx_2,
\end{eqnarray*}
which via a standard argument shows that
$$\|S_u(f)\|_\lozd  \ls \|\cn_P(f)\|_\lozd.$$
\end{proof}

\bigskip

{\bf Step 3: \ $\|f\|_{H^1_{\cn_P}( \rlz)}\ls
\|f\|_{H^1_{\mathcal R_P}(\rlz)}$} for $f\in H^1_{\mathcal R_P}(\rlz) \cap L^2(\rlz)$.

Let $f\in H^1_{\mathcal R_P}(\rlz) \cap L^2(\rlz)$. We define
 $u(t_1,t_2,x_1,x_2):= \Pto\Ptt f(x_1,x_2)$. For any $q\in(0,1)$, from \eqref{Moser},
for $r_1:=t_1/4$, $r_2:=t_2/4$ and for all $y_1, y_2$ with $|x_1-y_1|<r_1$, $|x_2-y_2|<r_2$,
we have
\begin{align*}
|u(t_1,t_2,y_1,y_2)|^q &\ls \frac1{t_1^2 t_2^2}\int_{B((x_1, \,t_1), \, r_1)}\int_{B((x_2, \,t_2), \, r_2)} |u(s_1,s_2,z_1,z_2)|^q\,dz_1dz_2ds_1ds_2\\
&\ls \frac1{t_1^2 t_2^2}\int_{B((x_1, \,t_1), \, r_1)}\int_{B((x_2, \,t_2), \, r_2)}\Big( \sup_{s_1>0,s_2>0} |u(s_1,s_2,z_1,z_2)|\Big)^q\,dz_1dz_2ds_1ds_2\\
&\ls \frac1{t_1 t_2}\int_{B(x_1, \,t_1)}\int_{B(x_2, \,t_2)}\mathcal{R}_Pf(z_1,z_2)^q\,dz_1dz_2\\
&\ls \mathcal M_s \Big(\big(\mathcal{R}_Pf\big)^q\Big)(x_1,x_2).
\end{align*}
 This implies that
$$   \cn_P^{1\over 4}(f)(x_1,x_2)\ls \bigg( \mathcal M_s \Big(\big(\mathcal{R}_Pf\big)^q\Big)(x_1,x_2)\bigg)^{1/q}.  $$
Note that in general
$$ \| \cn_P^{a}(f)\|_{L^1({\rlz})}\approx \| \cn_P^{b}(f)\|_{L^1({\rlz})} $$
for $a,b>0$, where the implicit constant depends only on $a$ and $b$. We have
$$ \| \cn_P(f)\|_{L^1(\rlz)}\ls \| \cn_P^{1\over4}(f)\|_{L^1(\rlz)} \ls \bigg\|\bigg( \mathcal M_s \Big(\big(\mathcal{R}_Pf\big)^q\Big)\bigg)^{1/q}\bigg\|_{L^1(\rlz)} \ls \|\mathcal R_P f\|_{L^1(\rlz)}.$$

\medskip
{\bf Step 4:}  $\|f\|_\hrp\ls \|f\|_\hrz$ for $f\in \hrz\cap L^2(\rlz)$.

This inequality follows from the subordination formula \eqref{subo}. The details of the proof here
are similar to the proof of Step 4 in \cite{dlwy2}.


\smallskip
{\bf Step 5:} $\|f\|_\hrz\le \|f\|_\hnz$ for $f\in \hnz\cap L^2(\rlz)$.
\medskip

This inequality is clear because $\crz_h f\le \cn_h f$.


\smallskip
{\bf Step 6:}
$
\|f\|_\hnz\ls \|f\|_\hap$ 
  for $f\in \hap\cap L^2(\rlz)$.

In order to show this property it suffices to prove that for every rectangular atom $a_R$, as in Definition \ref{def of product atom}, where $R=I\times J$ is a dyadic rectangle, and $\gz_1,\gz_2\ge 2$,
\begin{equation}\label{eqn:bdd of nontang maxi func on atom-1}
\int_{x_1\notin \gz_1 I}\inzf |\cn_h (a_R)(x_1, x_2)|\,dx_1dx_2\ls   |R|^{\frac{1}{2}} \|a_R\|_{L^2(\rlz)} \gz_1^{-1}
\end{equation}
and
\begin{equation}\label{eqn:bdd of nontang maxi func on atom-2}
\inzf\int_{x_2\notin \gz_2 J} |\cn_h (a_R)(x_1, x_2)|\,dx_1dx_2\ls  |R|^{\frac{1}{2}} \|a_R\|_{L^2(\rlz)} \gz_2^{-1}.
\end{equation}
In fact, these two inequalities follow from similar approaches and estimates from those in the proof \cite[Lemma 3.6]{DLY}. See also similar arguments in \cite[Equations (4.25) and (4.26)]{dlwy2}.

\section{Proofs of the inequalities \eqref{key3}}
\label{sec:third theorem}

In this section, we present the proofs of  inequalities \eqref{key3} under the condition that $\lz>1$.  Similar to \cite[Section 5]{dlwy2},  the main approach here is
to use the conjugate harmonic function estimates and the key tool is the Cauchy--Riemann type equations associated to
$S_\lz$. Since the techniques and concrete estimates here are quite different from those in \cite{dlwy2}, we provide the full details.

We begin with the following lemma which is a variant of \cite[Lemma 5]{ms}. We mention that
in the following lemma, we require $p\ge\lz/(2\lz-1)$ which originates from a technical method
from \cite{ms}.
\begin{lem}\label{l-subharmo F power}
Let $\lz\in(1, \fz)$, $p\in[\lz/(2\lz-1), \fz)$ and $F:=(u, v)$ with $u$ and $v$ satisfy the equation
\begin{equation}\label{CRS}
\displaystyle\left\{
    \begin{array}{ll}
      A_\lz u = \prz_t v, \\
      \prz_{t} u=A_{\lambda}^*v.
    \end{array}
  \right.
\end{equation}
If $|F|>0$, then
\begin{align}\label{classical subharm F power}
\Delta|F|^p:=\prz_t^2|F|^p+\prz^2_x|F|^p\ge 0.
\end{align}
\end{lem}

\begin{proof}
We use some ideas from \cite{ms}. Let $\prz_tF:=(\prz_t u, \prz_t v)$, $\prz_xF:=(\prz_x u, \prz_x v)$,
$F\cdot \prz_tF:= u\prz_t u+v\prz_t v$, $\cdots$. By \eqref{CRS}, we have
\begin{equation*}
\displaystyle\left\{
    \begin{array}{ll}
      \prz_{x} u-\prz_{t} v=\frac{\displaystyle \lz} {\displaystyle x} u, \\[4pt]
      \prz_{t} u+\prz_{x} v=-\frac{\displaystyle\lambda}{\displaystyle x}v,
    \end{array}
  \right.
\end{equation*}
from which we deduce that
\begin{equation*}
\displaystyle\left\{
    \begin{array}{ll}
      \prz_{t}^2 u+\prz_{x}^2 u=\frac{\displaystyle \lz^2-\lz} {\displaystyle x^2} u, \\[4pt]
      \prz_{t}^2 v+\prz_{x}^2 v=-\frac{\displaystyle\lz^2+\lz}{\displaystyle x^2}v.
    \end{array}
  \right.
\end{equation*}
Then we have 
\begin{align*}
\Delta|F|^p&=p|F|^{p-4}\lf\{(p-2)\lf[\lf(\prz_tF\cdot F\r)^2+\lf(\prz_xF\cdot F\r)^2\r]\r.\\
&\quad\lf.+|F|^2\lf[\prz_t^2 F\cdot F+
\prz_x^2 F\cdot F+|\prz_t F|^2+|\prz_x F|^2\r]\r\}\\
&=p|F|^{p-4}\lf\{(p-2)\lf[\lf(\prz_tF\cdot F\r)^2+\lf(\prz_xF\cdot F\r)^2\r]\r.\\
&\quad\lf.+|F|^2\lf[\frac{\lz^2-\lz}{x^2}u^2+\frac{\lz^2+\lz}{x^2}v^2+|\prz_t F|^2+|\prz_x F|^2\r]\r\}.
\end{align*}
It is obvious that $\Delta |F|^p\ge0$ when $p\ge 2$. Thus, it remains to consider the case when $p<2$.
Then \eqref{classical subharm F power} is equivalent with
\begin{align}\label{classical subharm F power-equiv}
\lf(\prz_tF\cdot F\r)^2+\lf(\prz_xF\cdot F\r)^2\le \frac1{2-p}|F|^2\lf[|\prz_t F|^2+|\prz_x F|^2+\frac{\lz^2-\lz}{x^2}u^2+\frac{\lz^2+\lz}{x^2}v^2\r].
\end{align}

Now let $M$ denote the matrix
\begin{equation*}
\left[
    \begin{array}{cc}
    u_x  &\ \ v_x \\
     u_t &\ v_t
   \end{array}
  \right].
\end{equation*}
Then by \eqref{CRS}, \eqref{classical subharm F power-equiv} can be translated to
\begin{align}\label{matrix inequa}
\lf|M[F]^2\r|\le \frac 1{2-p}|F|^2\lf[\|M\|^2+\frac{(\lz-1)(u_x-v_t)^2}{\lz}+ \frac{(\lz+1)(u_t+v_x)^2}{\lz}\r],
\end{align}
where $\|M\|$ denotes the `Hilbert-Schmidt' norm of the matrix $M$.

If we consider $F$ to be an arbitrary two-component vector, then \eqref{matrix inequa} becomes
\begin{align}\label{matrix norm inequa}
|M|^2\le \frac 1{2-p}\lf[\|M\|^2+ \frac{(\lz+1)(\prz_t u+\prz_x v)^2}{\lz}+\frac{(\lz-1)(\prz_x u-\prz_t v)^2}{\lz}\r],
\end{align}
where $|M|$ is the usual norm of the matrix $M$ as an operator.
Moreover, it suffices to show that
\begin{align}\label{matrix norm inequa-1}
\max\lf\{\prz_x u^2, \prz_t v^2\r\}\le \frac1{2-p}\lf[\prz_x u^2+\prz_t v^2+\frac{(\lz-1)(\prz_x u-\prz_t v)^2}{\lz}\r]
\end{align}
and
\begin{align}\label{matrix norm inequa-2}
\max\lf\{\prz_t u^2, \prz_x v^2\r\}\le \frac1{2-p}\lf[\prz_t u^2+\prz_x v^2+\frac{(\lz+1)(\prz_t u+\prz_x v)^2}{\lz}\r].
\end{align}
Arguing as in \cite[Equation (9.10)]{ms}, we see that for $\lz>1$, \eqref{matrix norm inequa-1} holds if $p\ge\frac{\lz}{2\lz-1}$,
and \eqref{matrix norm inequa-2} holds if  $p\ge\frac{\lz}{2\lz+1}$. Note that
$\frac{\lz}{2\lz+1}<\frac{\lz}{2\lz-1}$. We then conclude that when $p\ge\frac{\lz}{2\lz-1},$
\eqref{classical subharm F power} holds. This finishes the proof of Lemma \ref{l-subharmo F power}.
\end{proof}


For $f\in\lpzd$ with $p\in [1, \fz)$, and $t_1,\,t_2,\,x_1,\,x_2\in\R_+$, let
\begin{equation}\label{real and imaginary parts-1}
u(t_1,\,t_2,\,x_1,\,x_2):=\plzo\plzt f(x_1,x_2),\,\,\,\,v(t_1,\,t_2,\,x_1,\,x_2):=\qlzo\plzt f(x_1,x_2),
\end{equation}
and
\begin{equation}\label{real and imaginary parts-2}
w(t_1,\,t_2,\,x_1,\,x_2):=\plzo\qlzt f(x_1,x_2),\,\,\,\,z(t_1,\,t_2,\,x_1,\,x_2):=\qlzo\qlzt f(x_1,x_2),
\end{equation}
where $\qlzo$ and $\qlzt$ are the conjugates of $\plzo$ and $\plzt$, respectively. For the concrete definition, we refer to \eqref{eeee conjugate1} below and the following.
Moreover, define
\begin{equation}\label{radi maxi defn}
u^\ast(x_1, x_2):=\crz_{P}f(x_1, x_2)
\end{equation}
and
\begin{align}\label{sum of real and imagi parts}
F(t_1,\,t_2,\,x_1,\,x_2)&:=\lf\{[u(t_1,\,t_2,\,x_1,\,x_2)]^2+[v(t_1,\,t_2,\,x_1,\,x_2)]^2\r.\\
&\quad\quad+\lf.[w(t_1,\,t_2,\,x_1,\,x_2)]^2+[z(t_1,\,t_2,\,x_1,\,x_2)]^2\r\}^{\frac12}.\noz
\end{align}

Next we recall the conjugate Poisson kernel and establish an auxiliary result.

Suppose $f\in L^p(\R_+)$, $1\leq p<\infty$. According to \cite[(16.5) and (16.5)']{ms}, we define, for every $x,t>0$, the
conjugate $\qlz(f)$ by
\begin{align}\label{eeee conjugate1}
\qlz(f)(x)=\int_0^\infty \qlz(x,y)f(y)dy,
\end{align}
where
$$
 \qlz(x,y)=-(xy)^{1\over2}\int_0^\infty e^{-tz}z J_{\lz+{1\over2}}(xz)J_{\lz-{1\over2}}(yz)dz,\ t,x,y\in\R_+.
$$
\begin{lem}
For any function $f\in\loz\cap L^2(\R_+)$ with $\risz f\in\loz$,
\begin{align}\label{conjugacy}
\qlz f(x)=\inzf {\mathbb P}^{[\lz+1]}_t(x,y) \risz f(y)\,dy.
\end{align}
\end{lem}
\begin{proof}
%
Indeed, let $\mathcal H_\lz$ denote the Hankel transform defined by
\begin{align}
\mathcal H_\lz f (x):= \inzf\sqrt{xy} J_{\lz-1/2}(xy) f(y)\,dy,\quad x\in\R_+,
\end{align}
where $J_\nu$ is the Bessel function of the first kind and order $\nu$.
By using \cite[p.24]{emot} we can deduce that
\begin{align}\label{eeee conjugate2}
{\mathbb P}^{[\lz]}_t(x,y)=\int_0^\infty (xz)^{1\over2}J_{\lz-{1\over2}}(xz) (yz)^{1\over2} J_{\lz-{1\over2}}(yz)e^{-tz}dz, \ t,x,y\in\R_+.
\end{align}
Also, since $f\in L^2(\R_+)$, according to \cite[(16.8)]{ms}, we have that
\begin{align}\label{eeee conjugate3}
\risz(f)=-\mathcal{H}_{\lz+1}(\mathcal{H}_\lz(f)).
\end{align}
By using \eqref{eeee conjugate1} it follows that $\qlz(f)\in L^2(\R_+)$ for every $t>0$, and
\begin{align*}
\qlz(f)(x) &= - \int_0^\infty f(y) \int_0^\infty e^{-tz}(xz)^{1\over2} J_{\lz+{1\over2}}(xz)(yz)^{1\over2} J_{\lz-{1\over2}}(yz)dz\\
&= - \int_0^\infty e^{-tz}(xz)^{1\over2} J_{\lz+{1\over2}}(xz) \mathcal{H}_\lz(f)(z)dz\\
&= -\mathcal{H}_{\lz+1}\big(   e^{-tz}\mathcal{H}_\lz(f)(z) \big)(x),\ \ t,x\in\R_+.
\end{align*}
This interchange of integrals is justified because $f\in L^1(\mathbb R_+)$ and the function $z^{1\over2} J_\nu(z)$ is bounded on $\R_+$
when $\nu>-{1\over2}$.

On the other hand, by combining \eqref{eeee conjugate2}  and \eqref{eeee conjugate3}, since $\risz(f)\in L^1(\R_+)$, we also obtain that
\begin{align*}
{\mathbb P}^{[\lz+1]}_t \big(\risz f\big)(x) = -\mathcal{H}_{\lz+1}\big(   e^{-tz}\mathcal{H}_\lz(f)(z) \big)(x),\ \ t,x\in\R_+.
\end{align*}
Thus, \eqref{conjugacy} is proved.
\end{proof}

We now establish the following lemma with respect to the harmonic conjugate functions.
\begin{lem}\label{lem: F controled by riesz}
Let $f\in \horiz\cap L^2(\rlz)$, $F$ be as defined in  \eqref{sum of real and imagi parts}, $u,v$ as in
\eqref{real and imaginary parts-1} and $w,z$ as in \eqref{real and imaginary parts-2}
Then
\begin{eqnarray*}
&&\supd\dinrp  F(t_1,\,t_2,\,x_1,\,x_2) dx_1dx_2\ls  \|f\|_\horiz.
\end{eqnarray*}
\end{lem}

\begin{proof}
It suffices to show that
\begin{eqnarray}\label{F controled by riesz-u}
\supd\dinrp|u(t_1,\,t_2,\,x_1,\,x_2)|dx_1dx_2\ls \|f\|_\lozd,
\end{eqnarray}
\begin{eqnarray}\label{F controled by riesz-v}
\supd\dinrp|v(t_1,\,t_2,\,x_1,\,x_2)|dx_1dx_2\ls \| \riszo f\|_\lozd,
\end{eqnarray}
\begin{eqnarray}\label{F controled by riesz-w}
\supd\dinrp|w(t_1,\,t_2,\,x_1,\,x_2)|dx_1dx_2\ls \| \riszt f\|_\lozd,
\end{eqnarray}
and
\begin{eqnarray}\label{F controled by riesz-z}
\supd\dinrp|z(t_1,\,t_2,\,x_1,\,x_2)|dx_1dx_2\ls \| \riszo \riszt f\|_\lozd.
\end{eqnarray}

For \eqref{F controled by riesz-u}, note that from Corollary \ref{cor Poisson}, $\plz(x,y)$ has the standard Poisson upper bound (i.e. \eqref{Poisson}). Hence,  \eqref{F controled by riesz-u}
follows from a direct calculation by using \eqref{Poisson}.

Next, from \eqref{conjugacy} and by taking into account that the Poisson semigroup $\{ \mathbb{P}^{[\lz+1]}_t \}_{t>0}$
is uniformly bounded in $L^1(\R_+)$, we conclude that
\begin{eqnarray}\label{conjuga for poiss inte}
\lf\|\qlz f\r\|_\loz&=&\lf\| \mathbb{P}^{[\lz+1]}_t\big(\risz (f)\big)\r\|_\loz
\ls\|\risz f\|_\loz
\end{eqnarray}
and also
\begin{eqnarray*}
\dinrp|v(t_1,\,t_2,\,x_1,\,x_2)|dx_1dx_2\!
\ls\! \dinrp\lf|\qlzo f(x_1,\,x_2)\r|dx_1dx_2\ls  \| \riszo f\|_\lozd.
\end{eqnarray*}
This implies \eqref{F controled by riesz-v}. Similarly, we have \eqref{F controled by riesz-w}.  Finally, from
\eqref{eeee conjugate1}, we deduce that
\begin{eqnarray*}
z(t_1,\,t_2,\,x_1,\,x_2)=\inzf\inzf \mathbb P^{[\lz+1]}_{t_1}(x_1, y_1) \mathbb P^{[\lz+1]}_{t_2}(x_2, y_2) \riszo \riszt f(y_1, y_2)\,dy_1dy_2,
\end{eqnarray*}
which shows \eqref{F controled by riesz-z} immediately.
This finishes the proof of Lemma \ref{lem: F controled by riesz}.
\end{proof}

\smallskip
\begin{proof}[{\bf Proof of Inequalities \eqref{key3}}]

We first show that for any $f\in H^1_{S}(\rlz)$,
\begin{align}\label{R leq S}
\|f\|_{H^1_{Riesz}(\rlz)}\ls \|f\|_{H^1_{S}(\rlz)}.
\end{align}
To see this,  it suffices to prove that
\begin{align*}
\|f\|_\lozd+\| \riszo f\|_\lozd+ \| \riszt f\|_\lozd+ \| \riszo \riszt f\|_\lozd \ls \|f\|_{H^1_{S}(\rlz)}.
\end{align*}
We point out that since the kernel $\Ht(x,y)$ of the heat semigroup $\{\Ht\}_{t>0}$ satisfies the Gaussian estimate \eqref{Ga} (see Theorem \ref{thm Gaussian upper bound}),
the inequality $\| \riszo \riszt f\|_\lozd \ls \|f\|_{H^1_{\sbz}(\rlz)}$
follows proceeding as in \cite[Theorem 5.1]{DLY}.
And we consider
$ \riszo $ as $ \riszo\otimes Id_2$ and $ \riszt$ as $Id_1\otimes \riszt$, where we use $Id_1$ and $Id_2$ to denote the identity operator on $L^2(\R_+)$. Then following the
proof of \cite[Theorem 5.1]{DLY}, we also obtain
$\| \riszo  f\|_\lozd \ls \|f\|_{H^1_{\sbz}(\rlz)}$ and $\|  \riszt f\|_\lozd \ls \|f\|_{H^1_{\sbz}(\rlz)}$.
As a consequence, we obtain that \eqref{R leq S} holds.

Next, assume that $f\in H^1_{Riesz}(\rlz)$.
We now show that
\begin{align}\label{Rp leq R}
\|f\|_{H^1_{\mathcal{R}_{P}}(\rlz)}\ls \|f\|_{H^1_{Riesz}(\rlz)}.
\end{align}
To this end, based on Lemma \ref{lem: F controled by riesz}, it remains to prove that
\begin{equation}\label{radia max cont by holomo fun}
\|f\|_{H^1_{\mathcal{R}_{P}}(\rlz)}=\|u^\ast\|_\lozd\ls{\displaystyle\supd\dinrp} F(t_1,\,t_2,\,x_1,\,x_2)\,\dmzo\dmzt,
\end{equation}
where $u^\ast$ and $F$ are as in \eqref{radi maxi defn} and \eqref{sum of real and imagi parts}.
We first claim that  we only need to show that for $p\in \big(\frac{\lz}{2\lz-1}, 1\big)$ and $\ez_1,\,t_1,\,\ez_2,\, t_2,\,x_1,\,x_2\in\R_+$,
\begin{equation}\label{F subharmonic}
\widetilde F^p(\ez_1+t_1,\,\ez_2+ t_2,\,x_1,\,x_2)\ls P_{t_1}P_{t_2}\lf(\widetilde F^p(\ez_1,\,\ez_2,\,\cdot,\,\cdot)\r)(x_1,x_2),
\end{equation}
where $P_t$ is the classical Poisson kernel and $\widetilde F(t_1,t_2,x_1,x_2)$
is the even extension of $F(t_1,t_2,x_1,x_2)$ in $x_1$ and $x_2$ to $\R$, respectively, that is,
 \[
    \widetilde F(t_1,t_2,x_1,x_2)=\left\{
                \begin{array}{ll}
                  F(t_1,t_2,x_1,x_2),\quad &x_1>0,\,x_2>0;\\
                  F(t_1,t_2,-x_1,x_2),\quad &x_1<0,\,x_2>0;\\
                    F(t_1,t_2,x_1,-x_2),\quad& x_1>0,\,x_2<0;\\
                     F(t_1,t_2,-x_1,-x_2),\quad& x_1<0,\,x_2<0.
             \end{array}
              \right.
  \]
Indeed, by Lemma \ref{lem: F controled by riesz}, we see that
 $\{\widetilde F^p(\ez_1,\,\ez_2,\,\cdot,\,\cdot)\}_{\ez_1,\,\ez_2>0}$ is uniformly
 bounded on $L^r(\mathbb R\times\mathbb R)$. Since $L^r(\mathbb R\times\mathbb R)$  is reflexive, there exist two  sequences $\{\ez_{1,\,k}\}$, $\{\ez_{2,\,j}\}\downarrow0$ and
 $h\in L^r(\mathbb R\times\mathbb R)$ such that $\{\widetilde F^p(\ez_{1,\,k},\,\ez_{2,\,j},\,\cdot,\,\cdot)\}_{\ez_{1,\,k},\,\ez_{2,\,j}>0}$ converges weakly to $h$
 in $L^r(\mathbb R\times\mathbb R)$ as $k,\,j\to\fz$. Moreover, by H\"older's inequality, we see that
\begin{align}\label{norm of weak limit}
\|h\|_{L^r(\mathbb R\times\mathbb R)}^r
&=\bigg\{\dsup_{\|g\|_{L^{r'}(\mathbb R\times\mathbb R)}\le1}\lf|\dint_{\mathbb R\times\mathbb R} g(x_1,x_2)h(x_1,x_2)\,dx_1dx_2\r|\bigg\}^r\\
&=\bigg\{\dsup_{\|g\|_{L^{r'}(\mathbb R\times\mathbb R)}\le1}\lim_{\gfz{k\to\fz}{j\to\fz}}
\lf|{\displaystyle\dint_{\mathbb R\times\mathbb R}} g(x_1, x_2)\widetilde F^p(\ez_{1,\,k},\,\ez_{2,\,j},\,\cdot,\,\cdot)\,dx_1dx_2\r|\bigg\}^r\noz\\
&\le \dlimsup_{\gfz{k\to\fz}{j\to\fz}}\big\|\widetilde F^p(\ez_{1,\,k},\,\ez_{2,\,j},\,\cdot,\,\cdot)\big\|_{L^r(\mathbb R\times\mathbb R)}^r\noz\\
&\ls\sup_{t_1>0,\,t_2>0}\dinrp F(t_1,\,t_2,\,x_1,\,x_2)\,dx_1dx_2.\noz
\end{align}
Since $\widetilde F$ is continuous in $t_1$ and $t_2$, for any $x_1,\,x_2\in\R_+$,
$$\widetilde F^p(t_1+\ez_{1,\,k}, t_2+\ez_{2,\,j}, x_1, x_2)\to \widetilde F^p(t_1, t_2, x_1, x_2)$$ as $k,\,j\to\fz$.
Observe that for each $x_1,\,x_2\in \R_+$,
$$P_{t_1}P_{t_2}(\widetilde F^p(\ez_{1,\,k},\,\ez_{2,\,j},\,\cdot,\,\cdot))(x_1, x_2)\to P_{t_1}P_{t_2}(h)(x_1, x_2)$$
as $k,\,j\to\fz$.
Thus, by these facts and \eqref{F subharmonic}, we have that
for any $t_1,\,t_2,\,x_1,\,x_2\in\R_+$,
\begin{eqnarray*}
\widetilde F^p(t_1, t_2, x_1, x_2)&=&\dlim_{\gfz{k\to\fz}{j\to \fz}}\widetilde F^p(t_1+\ez_{1,\,k}, t_2+\ez_{2,\,j}, x_1, x_2)\\
&\ls&\dlim_{\gfz{k\to\fz}{j\to\fz}}P_{t_1}P_{t_2}(\widetilde F^p(\ez_{1,\,k},\,\ez_{2,\,j},\,\cdot,\,\cdot))(x_1, x_2)\\
&=&P_{t_1}P_{t_2}(h)(x_1, x_2).
\end{eqnarray*}
Therefore, for any $x_1,\,x_2\in\R_+$,
\begin{equation*}
[u^\ast(x_1, x_2)]^p\le\sup_{t_1>0,\,t_2>0}F^p(t_1,t_2, x_1, x_2)\ls\mathcal M_{R_P}(h)(x_1, x_2),
\end{equation*}
where  $\mathcal M_{R_P}$ is the classical radial maximal function.
By this together with $r:=1/p$, the $L^r(\mathbb R\times\mathbb R)$-boundedness of $\mathcal M_{R_P}$ and \eqref{norm of weak limit}, we then have
\begin{align*}
\|u^\ast\|_\lozd&\ls \Big\|\mathcal M_{R_P}(h)\Big\|^r_{L^r(\mathbb R\times\mathbb R)}
\ls{\sup_{t_1>0,\,t_2>0}\dinrp} F(t_1,\,t_2,\,x_1,\,x_2)\,dx_1dx_2,
\end{align*}
which implies that \eqref{radia max cont by holomo fun}. Thus the claim holds.

Now we prove \eqref{F subharmonic}.
 Observe that for any fixed $t_2,\,x_2\in\R_+$, $u,\,v$ and $w,\,z$ respectively satisfy the Cauchy--Riemann equations for $t_1$ and
$x_1$,
and for any fixed $t_1, \,x_1\in \R_+$, $u,\,w$ and $v,\,z$ respectively satisfy the Cauchy--Riemann equations for $t_2$ and
$x_2$. That is,
\begin{equation}\label{CR equ-1}
\displaystyle\left\{
    \begin{array}{ll}
      \prz_{x_1} u-\prz_{t_1} v=\frac\lz {x_1} u, \\
      \prz_{t_1} u+\prz_{x_1} v=-\frac{\lambda}{x_1}v;
    \end{array}
  \right.
\hspace{0.8cm}
\left\{
    \begin{array}{ll}
    \prz_{x_1} w-\prz_{t_1} z=\frac\lz{x_1}w, \\
      \prz_{t_1} w+\prz_{x_1} z=-\frac{\lambda}{x_1}z;
   \end{array}
  \right.
\end{equation}
and
\begin{equation}\label{CR equ-2}
\left\{
    \begin{array}{ll}
      \prz_{x_2} u-\prz_{t_2} w=\frac\lz{x_2}u, \\
      \prz_{t_2} u+\prz_{x_2} w=-\frac{\lambda}{x_2}w;
    \end{array}
  \right.
\hspace{0.8cm}
\left\{
    \begin{array}{ll}
    \prz_{x_2} v-\prz_{t_2} z=\frac\lz{x_2}v, \\
      \prz_{t_2} v+\prz_{x_2} z=-\frac{\lambda}{x_2}z.
   \end{array}
  \right.
\end{equation}

For fixed $t_2,\, x_2\in\R_+$, let
$$ F_1(t_1,\,t_2,\,x_1,\,x_2):=\lf\{[u(t_1,\,t_2,\,x_1,\,x_2)]^2+[v(t_1,\,t_2,\,x_1,\,x_2)]^2\r\}^{\frac12},$$
where $t_1,\,x_1\in \R_+$. For the moment, we fix $t_2$, $x_2$ and regard $F_1$ as a function of $t_1$ and $x_1$. Then
we claim that:

(1)  $F_1^p$ is subharmonic in the classical sense for $p\in (\frac{\lz-1}{2\lz-1}, 1]$.

Actually, this follows from \eqref{CR equ-1}, Lemma \ref{l-subharmo F power} 
and \cite[Theorem 4.4]{sw}.

\smallskip

(2) for almost every $t_2\in \R_+$ and almost every $x_2\in \R_+$,
\begin{equation}\label{subharm fun  norm bdd}
\sup_{t_1>0}\inzf \lf[F^p_1(t_1,\,t_2,\,x_1,\,x_2)\r]^r\dmzo\leq \sup_{t_1>0}\inzf F(t_1,\,t_2,\,x_1,\,x_2)\dmzo<\fz.
\end{equation}

To prove \eqref{subharm fun  norm bdd}, we fix $t_2,x_2\in (\R_+)$. Then we define
$$ U(t_1,x_1):=u(t_1,t_2,x_1,x_2)= \mathbb{P}_{t_1}^{[\lz]} \mathbb{P}_{t_2}^{[\lz]}(f)(x_1,x_2),\ t_1,x_1\in\R_+ $$
and
$$ V(t_1,x_1):=v(t_1,t_2,x_1,x_2)= \mathbb{Q}_{t_1}^{[\lz]} \mathbb{P}_{t_2}^{[\lz]}(f)(x_1,x_2),\ t_1,x_1\in\R_+. $$
We note that $$  V(t_1,x_1)=  \mathbb{P}_{t_1}^{[\lz+1]} \mathbb{P}_{t_2}^{[\lz]}\big( R_{S_\lz,1}(f)\big)(x_1,x_2),\ t_1,x_1\in\R_+. $$
Since the Poisson semigroup $\{ \mathbb{P}_{t_1}^{[\lz]} \}_{t_1>0}$ is uniformly bounded in $L^1(\R_+)$, we get
$$ \sup_{t_1>0}\int_0^\infty |U(t_1,x_1)|dx_1\leq C\int_0^\infty \big| \mathbb{P}_{t_2}^{[\lz]}(f(x_1,\cdot))(x_2) \big|dx_1, $$
and
$$ \sup_{t_1>0}\int_0^\infty |V(t_1,x_1)|dx_1\leq C\int_0^\infty \big| \mathbb{P}_{t_2}^{[\lz]}\big(R_{S_\lz,1}(f)(x_1,\cdot)\big)(x_2) \big|dx_1. $$
Then we further have
\begin{align*}
\sup_{t_2>0}\int_0^\infty\sup_{t_1>0}\int_0^\infty |u(t_1,t_2,x_1,x_2)|dx_1dx_2 &\leq C\sup_{t_2>0}\int_0^\infty\int_0^\infty \big| \mathbb{P}_{t_2}^{[\lz]}(f(x_1,\cdot))(x_2) \big|dx_1dx_2\\
&\leq C\|f\|_{L^1(\rlz)}
\end{align*}
and
\begin{align*}
\sup_{t_2>0}\int_0^\infty\sup_{t_1>0}\int_0^\infty |v(t_1,t_2,x_1,x_2)|dx_1dx_2 &\leq C\sup_{t_2>0}\int_0^\infty\int_0^\infty \big| \big(R_{S_\lz,1}(f)(x_1,\cdot)\big)(x_2) \big|dx_1dx_2\\
&\leq C\|R_{S_\lz,1}(f)\|_{L^1(\rlz)}.
\end{align*}
We deduce that for every $t_2>0$ there exists $W_{t_2}\subset \R_+$ such that $|W_{t_2}|=0$ and
$$\sup_{t_1>0} \int_0^\infty |u(t_1,t_2,x_1,x_2)|dx_1<\infty\quad{\rm and}\quad \sup_{t_1>0} \int_0^\infty |v(t_1,t_2,x_1,x_2)|dx_1<\infty$$
for every $x_2\in \R_+\backslash W_{t_2}$. Hence, there exist $W\subset \R_+$ with $|W|=0$ such that
$$\sup_{t_1>0} \int_0^\infty |u(t_1,t_2,x_1,x_2)|dx_1<\infty\quad{\rm and}\quad \sup_{t_1>0} \int_0^\infty |v(t_1,t_2,x_1,x_2)|dx_1<\infty$$
for every $x_2\in \R_+\backslash W$ and $t_2\in \R_+\setminus \mathbb{ Q}$, where we use $\mathbb Q$ to denote the set of all  rational numbers. This shows that  \eqref{subharm fun  norm bdd} holds.


From the claims (1) and (2) (for $x_2\in \R_+\backslash W$ and $t_2\in \R_+\setminus \mathbb{ Q}$), and from \cite[Theorem 4.6]{sw},  it follows that  
\begin{align}\label{harmo contr-F1}
\widetilde F^p_1(\ez_1+t_1,\,t_2,\,x_1,\,x_2)\le P_{t_1}\lf(\widetilde F^p_1(\ez_1,\,t_2,\,\cdot,\, x_2)\r)(x_1)
\end{align}
for every $\ez_1,t_1\in\R_+,$ $x_1\in \R$, where $\widetilde F_1$ is the even extension of $F_1$ in $x_1$ and $x_2$.

Similarly, let
$$F_2(t_1,\,t_2,\,x_1,\,x_2):=\lf\{[w(t_1,\,t_2,\,x_1,\,x_2)]^2+[z(t_1,\,t_2,\,x_1,\,x_2)]^2\r\}^{\frac12}$$
and $\widetilde F_2$ is the even extension of $F_2$ in $x_1$ and $x_2$.
By Lemma \ref{l-subharmo F power}  and \eqref{subharm fun  norm bdd} with
$F_1$ replaced by $F_2$ therein, \cite[Theorems 4.4 and 4.6]{sw} again, we have that
for any $\ez_1, t_1,\,t_2\in\R_+$, $x_1$, $x_2\in\R$,
\begin{align}\label{harmo contr-F2}
\widetilde F^p_2(\ez_1+t_1,\,t_2,\,x_1,\,x_2)\le P_{t_1}\lf(\widetilde F^p_2(\ez_1,\,t_2,\,\cdot,\, x_2)\r)(x_1).
\end{align}
Observe that for any $t_1,\,t_2\in \R_+$ and $x_1,\,x_2\in \R$,
\begin{equation*}
\widetilde F(t_1,\,t_2,\,x_1,\,x_2)\approx \sum_{i=1}^2\widetilde F_i(t_1,\,t_2,\,x_1,\,x_2).
\end{equation*}
By this fact, \eqref{harmo contr-F1} and \eqref{harmo contr-F2}, we have that
\begin{equation}\label{harmo contr-F}
\widetilde F^p(\ez_1+t_1,\,t_2,\,x_1,\,x_2)\ls P_{t_1}\lf(\widetilde F^p(\ez_1,\,t_2,\,\cdot,\,\cdot)\r)(x_1,x_2).
\end{equation}
Moreover, from \eqref{CR equ-2}, Lemma \ref{lem: F controled by riesz}, Lemma \ref{l-subharmo F power} and \cite[Theorems 4.4 and 4.6]{sw},
we also deduce that
\begin{equation*}
\widetilde F^p(t_1,\,\ez_2+ t_2,\,x_1,\,x_2)\ls P_{t_2}\lf(\widetilde F^p(t_1,\,\ez_2,\,\cdot,\,\cdot)\r)(x_1,x_2).
\end{equation*}
Now by this and \eqref{harmo contr-F}, we conclude that
\begin{eqnarray*}
\widetilde F^p(\ez_1+t_1,\,\ez_2+ t_2,\,x_1,\,x_2)&&\ls P_{t_1}\lf(\widetilde F^p(\ez_1,\,\ez_2+ t_2,\,\cdot,\,\cdot)\r)(x_1,x_2)\\
&&\ls P_{t_1}P_{t_2}\lf(\widetilde F^p(\ez_1,\,\ez_2,\,\cdot,\,\cdot)\r)(x_1,x_2).\noz
\end{eqnarray*}
This implies \eqref{F subharmonic}, and hence finishes the proof of \eqref{key3}.
\end{proof}

\section{Proof of second main result: Theorem \ref{thm main 2}}
\label{sec:app}


We recall the Telyakovski\'i transform, which is defined for any locally integrable function $f: \mathbb R_+\to \mathbb R$ by
\begin{align}
\mathcal T_{\mathbb R_+}f(x) =p.v.\int_0^{x\over2} {f(x-t) -f(x+t)\over t}dt =p.v. \int_{x\over2}^{3x\over2} {f(t) \over x-t}dt,
\end{align}
where the integral is defined in the Cauchy principal value sense. The operator $\mathcal T_{\mathbb R_+}$ resembles the Hilbert transform
$\mathcal H$ defined as
\begin{align*}
\mathcal H f(x)=p.v.\int_0^{\infty} {f(x-t) -f(x+t)\over t}dt =p.v.\int_{-\infty}^{\infty} {f(t) \over x-t}dt.
\end{align*}
Here we omit the usual constant $1/\pi$ factor in the above definitions.

Next we consider the setting of $\R_+\times\R_+$. We use $\mathcal T_{\mathbb R_+,1}$ to denote the Telyakovski\'i transform on the first variable and
$\mathcal T_{\mathbb R_+,2}$ the second. Similarly for the notation of $\mathcal H_1$ and $\mathcal H_2 $.  Now, as stated in the introduction, we define the product Hardy space in terms of Telyakovski\'i transforms.
\begin{defn}
Let
$ H^1_{\mathcal T}(\rlz)$ be the completion of $$\{f\in L^1(\rlz)\cap L^2(\rlz):\ \mathcal T_{\mathbb R_+,1}f, \mathcal T_{\mathbb R_+,2}f, \mathcal T_{\mathbb R_+,1}\mathcal T_{\mathbb R_+,2}f\ \in L^1(\rlz)\} $$
with respect to  the norm
$$ \|f\|_{H^1_{\mathcal T}(\rlz)}:= \|f\|_{L^1(\rlz)}+ \|\mathcal T_{\mathbb R_+,1}f\|_{L^1(\rlz)}+ \|\mathcal T_{\mathbb R_+,2}f\|_{L^1(\rlz)}+ \|\mathcal T_{\mathbb R_+,1}\mathcal T_{\mathbb R_+,2}f\|_{L^1(\rlz)}. $$
\end{defn}


Then we have the following structure theorem.
\begin{thm}
$H^1_{\mathcal T}(\rlz)$ is isomorphic to the subspace of odd functions (as defined in \eqref{odd}) in $H^1(\R\times\R)$, which is the standard Chang--Fefferman product Hardy space.
\end{thm}
\begin{proof}
Suppose $f\in H^1_{\mathcal T}(\rlz)\cap L^2(\rlz)$. Let $f_o$ be the product odd extension of $f$ as defined in \eqref{odd}. We now show that
$f_o\in H^1(\R\times\R)$.

To see this, recalling the characterization of $H^1(\R\times\R)$ via double Hilbert transforms, it suffices to show that
$ \mathcal H_1f_o, \mathcal H_2f_o, \mathcal H_1 \mathcal H_2f_o \in L^1(\R\times\R)$.
Since the function $\mathcal H_1f_o$ ($\mathcal H_2f_o$ resp.) is an odd function in the second
(first resp.) variable and
even function in the first (second resp.) variable,  and $\mathcal H_1 \mathcal H_2f_o $
is  an even function in terms of the first variable and second variable, it suffices to show that
 $ \mathcal H_1f_o, \mathcal H_2f_o, \mathcal H_1 \mathcal H_2f_o \in L^1(\rlz)$.

As for $ \mathcal H_1f_o$, we have by definition for every $x_1,x_2>0$,
\begin{align}\label{equ H1}
\mathcal H_1f_o(x_1,x_2) -\mathcal T_{\mathbb R_+,1}f(x_1,x_2) = 2I_1(f)(x_1,x_2)+2I_2(f)(x_1,x_2) - I_3(f)(x_1,x_2),
\end{align}
where
\begin{align*}
I_1(f)(x_1,x_2)&= \int_0^{x_1\over2} f(t_1,x_2) {t_1\over x_1^2-t_1^2}dt_1\\
I_2(f)(x_1,x_2)&= \int_{3x_1\over2}^\infty f(t_1,x_2) {t_1\over x_1^2-t_1^2}dt_1\\
I_3(f)(x_1,x_2)&= \int_{x_1\over2}^{3x_1\over2} f(t_1,x_2) {1\over x_1+t_1}dt_1.
\end{align*}
A direct calculation shows that
\begin{align*}
\|I_1(f)\|_{L^1(\rlz)} &\leq \int_0^\infty\int_0^\infty \int_0^{x_1\over2} |f(t_1,x_2)| {t_1\over x_1^2-t_1^2}dt_1dx_1dx_2=\ln\sqrt3 \|f\|_{L^1(\rlz)}, \\
\|I_2(f)\|_{L^1(\rlz)} &\leq \int_0^\infty\int_0^{\infty} \int_{3x_1\over2}^\infty  |f(t_1,x_2)| {t_1\over |x_1^2-t_1^2|}dt_1dx_1dx_2=\ln\sqrt5 \|f\|_{L^1(\rlz)}, \\
\|I_3(f)\|_{L^1(\rlz)} &\leq \int_0^\infty\int_0^\infty \int_{x_1\over2}^{3x_1\over2} |f(t_1,x_2)| {1\over x_1+t_1}dt_1dx_1dx_2=\ln(5/3) \|f\|_{L^1(\rlz)}.
\end{align*}
Hence we obtain that
\begin{align}\label{eee H1}
\|\mathcal H_1f_o\|_{L^1(\rlz)}
&\leq \|\mathcal T_{\mathbb R_+,1}f\|_{L^1(\rlz)}+C\|f\|_{L^1(\rlz)}.
\end{align}
As for $ \mathcal H_2f_o$, note that by definition, for $x_1,x_2>0$,
\begin{align}\label{equ H2}
\mathcal H_2f_o(x_1,x_2) -\mathcal T_{\mathbb R_+,2}f(x_1,x_2) = 2J_1(f)(x_1,x_2)+2J_2(f)(x_1,x_2) - J_3(f)(x_1,x_2),
\end{align}
where
\begin{align*}
J_1(f)(x_1,x_2)&= \int_0^{x_2\over2} f(x_1,t_2) {t_2\over x_2^2-t_2^2}dt_2\\
J_2(f)(x_1,x_2)&= \int_{3x_2\over2}^\infty f(x_1,t_2) {t_2\over |x_2^2-t_2^2|}dt_2\\
J_3(f)(x_1,x_2)&= \int_{x_2\over2}^{3x_2\over2} f(x_1,t_2) {1\over x_2+t_2}dt_2.
\end{align*}
Again, a direct calculation shows that
\begin{align*}
\|J_1(f)\|_{L^1(\rlz)} &\leq \int_0^\infty\int_0^\infty \int_0^{x_2\over2} |f(x_1,t_2)| {t_2\over x_2^2-t_2^2}dt_2dx_2dx_1=\ln\sqrt3 \|f\|_{L^1(\rlz)}, \\
\|J_2(f)\|_{L^1(\rlz)} &\leq \int_0^\infty\int_0^\infty \int_{3x_2\over2}^\infty |f(x_1,t_2)| {t_2\over |x_2^2-t_2^2|}dt_2dx_2dx_1=\ln\sqrt5 \|f\|_{L^1(\rlz)}, \\
\|J_3(f)\|_{L^1(\rlz)} &\leq \int_0^\infty\int_0^\infty \int_{x_2\over2}^{3x_2\over2} |f(x_1,t_2)| {1\over x_2+t_2}dt_2dx_2dx_1=\ln(5/3) \|f\|_{L^1(\rlz)}.
\end{align*}
And, hence we obtain that
\begin{align}\label{eee H2}
\|\mathcal H_2f_o\|_{L^1(\rlz)}
&\leq \|\mathcal T_{\mathbb R_+,2}f\|_{L^1(\rlz)}+C\|f\|_{L^1(\rlz)}.
\end{align}


As for $\mathcal H_1 \mathcal H_2f_o$, note that by definition, for $x_1>0$, $x_2>0$,
\begin{align}\label{eee HH}
&\mathcal H_1\mathcal H_2f_o(x_1,x_2) -\mathcal T_{\mathbb R_+,1}H_2f_o(x_1,x_2)\\
&= 2 \int_0^{x_1\over2} H_2f_o(t_1,x_2) {t_1\over x_1^2-t_1^2}dt_1+2 \int_{3x_1\over2}^\infty H_2f_o(t_1,x_2) {t_1\over x_1^2-t_1^2}dt_1\nonumber\\
&\quad -  \int_{x_1\over2}^{3x_1\over2} \mathcal{ H}_2f_o(t_1,x_2) {1\over x_1+t_1}dt_1\nonumber\\
&= 2I_1\Big(\mathcal{ H}_2f_o\Big)(x_1,x_2) +  2I_2\Big(\mathcal{ H}_2f_o\Big)(x_1,x_2) -  I_3\Big(\mathcal{ H}_2f_o\Big)(x_1,x_2). \nonumber
\end{align}
According to the estimates of $I_1$, $I_2$ and $I_3$ above, we have that
the $L^1(\rlz)$ norm of the three terms in the right-hand side of \eqref{eee HH} is bounded by
$ \|\mathcal{ H}_2f_o\|_{L^1(\rlz)} $, which is further controlled by $ \|\mathcal T_{\mathbb R_+,2}f\|_{L^1(\rlz)}+C\|f\|_{L^1(\rlz)}$ as showed in \eqref{eee H2}.  Thus, it is easy to see that
\begin{align*}
\|\mathcal H_1\mathcal H_2f_o\|_{L^1(\rlz)}
 \leq \|\mathcal T_{\mathbb R_+,1}\mathcal{ H}_2f_o\|_{L^1(\rlz)} +  \|\mathcal T_{\mathbb R_+,2}f\|_{L^1(\rlz)}+C\|f\|_{L^1(\rlz)}.
\end{align*}

Moreover, for the term $\mathcal T_{\mathbb R_+,1}\mathcal{ H}_2f_o(x_1,x_2)$, we have
\begin{align}\label{eee HH2}
\mathcal T_{\mathbb R_+,1}H_2f_o(x_1,x_2) - \mathcal T_{\mathbb R_+,1}\mathcal T_{\mathbb R_+,2}f(x_1,x_2)&= 2 \int_{x_1\over2}^{3x_1\over2} \int_0^{x_2\over2} f(t_1,t_2) {t_2\over x_2^2-t_2^2}{1\over x_1-t_1} dt_2dt_1\\
&\quad +2\int_{x_1\over2}^{3x_1\over2} \int_{3x_2\over2}^\infty f(t_1,t_2) {t_2\over x_2^2-t_2^2}{1\over x_1-t_1}dt_2dt_1\nonumber\\
&\quad -  \int_{x_1\over2}^{3x_1\over2} \int_{x_2\over2}^{3x_2\over2} f(t_1,t_2) {1\over x_2+t_2}{1\over x_1-t_1}dt_2dt_1\nonumber\\
&=: K_1+K_2+K_3.\nonumber
\end{align}
We now consider $K_1$.
First note that for $f\in H^1_{\mathcal T}(\rlz)\cap L^2(\rlz)$,
$$K_1=2 \int_0^{x_2\over2} \mathcal T_{\mathbb R_+,1}f(x_1,t_2) {t_2\over x_2^2-t_2^2} dt_2 =2 J_1\Big(\mathcal T_{\mathbb R_+,1}f\Big)(x_1,x_2).$$
In fact, this follows from the facts that $\mathcal T_{\mathbb R_+,1}$
is bounded on $L^2(\rlz)$ (see \cite[Lemma 1]{am}) and that $J_1$ is also bounded on $L^2(\rlz)$, which follows from a direct calculation.

Then, by noting that $\mathcal T_{\mathbb R_+,1}f\in L^1(\rlz)$
and according to the estimates of $J_1$ above, we have
$$ \|K_1\|_{L^1(\rlz)} \leq C  \|\mathcal T_{\mathbb R_+,1}f\|_{ L^1(\rlz)}.$$
Again, according to the estimates of $J_2$ and $J_3$ above, we have that
the $L^1(\rlz)$ norms of $K_2$ and $K_3$ are both bounded by $C  \|\mathcal T_{\mathbb R_+,1}f\|_{ L^1(\rlz)}$. Here, the singular integrals must be understood as principal values. Thus, it is easy to see that
\begin{align*}
\|\mathcal T_{\mathbb R_+,1}\mathcal H_2f_o\|_{L^1(\rlz)}
 \leq \|\mathcal T_{\mathbb R_+,1}\mathcal T_{\mathbb R_+,2}f\|_{L^1(\rlz)} + C\|\mathcal T_{\mathbb R_+,1}f\|_{L^1(\rlz)}.
\end{align*}
Combining these estimates, we have
\begin{align}\label{eee H3}
&\|\mathcal H_1\mathcal H_2f_o\|_{L^1(\rlz)}\\
 &\leq \|\mathcal T_{\mathbb R_+,1}\mathcal T_{\mathbb R_+,2}f\|_{L^1(\rlz)} + C\|\mathcal T_{\mathbb R_+,1}f\|_{L^1(\rlz)}+C\|\mathcal T_{\mathbb R_+,2}f\|_{L^1(\rlz)}+C\|f\|_{L^1(\rlz)}.\nonumber
\end{align}
Hence, combining the estimates in \eqref{eee H1}, \eqref{eee H2} and \eqref{eee H3}, we obtain that
$ \mathcal H_1f_o, \mathcal H_2f_o, \mathcal H_1 \mathcal H_2f_o \in L^1(\rlz)$, which in turn gives
$ \mathcal H_1f_o, \mathcal H_2f_o, \mathcal H_1 \mathcal H_2f_o \in L^1(\R\times\R)$, i.e.,
$f_o\in H^1(\R\times\R)$.

Conversely, based on the same estimates above, we can also obtain that for $f\in L^1(\rlz)\cap L^2(\rlz)$, if $f_o\in H^1(\R\times\R)$ then we have the following estimates:
\begin{align}\label{eee r H1}
\|\mathcal T_{\mathbb R_+,1}f\|_{L^1(\rlz)}
&\leq \|\mathcal H_1f_o\|_{L^1(\rlz)}+C\|f\|_{L^1(\rlz)},
\end{align}
which follows from the equality \eqref{equ H1} and the estimates for \eqref{eee H1};
\begin{align}\label{eee r H2}
\|\mathcal T_{\mathbb R_+,2}f\|_{L^1(\rlz)}
&\leq \|\mathcal H_2f_o\|_{L^1(\rlz)}+C\|f\|_{L^1(\rlz)},
\end{align}
which follows from the equality \eqref{equ H2} and the estimates for \eqref{eee H2};
and
\begin{align}\label{eee r HH}
&\|\mathcal T_{\mathbb R_+,1}\mathcal T_{\mathbb R_+,2}f\|_{L^1(\rlz)} \\
&\leq C(\|\mathcal H_1\mathcal H_2f\|_{L^1(\rlz)}+\|\mathcal H_1f\|_{L^1(\rlz)}+\|\mathcal H_2f\|_{L^1(\rlz)}+\|f\|_{L^1(\rlz)}),\nonumber
\end{align}
which follows from the equalities \eqref{eee HH} and \eqref{eee HH2}, and from the estimates for \eqref{eee H3}.

Estimates \eqref{eee r H1},  \eqref{eee r H2} and  \eqref{eee r HH}  combined together give that $f\in H^1_{\mathcal T}(\rlz)$.
\end{proof}

\begin{thm}
The Hardy spaces $H^1_{Riesz}(\rlz)$ and $H^1_{\mathcal{T}}(\rlz)$ coincide and they have equivalent norms.
\end{thm}
\begin{proof}

Now suppose  $f\in H^1_{\mathcal{T}}(\rlz)\cap L^2(\rlz)$. We will show that $f$ is in $H^1_{Riesz}(\rlz)$, i.e., we need to verify that
$R_{S_\lz,1}(f), R_{S_\lz,2}(f)  $ and $R_{S_\lz,1}R_{S_\lz,2}(f)$ are all in $L^1(\rlz)$.

One observes that the Riesz transform $R_{S_\lz,1}f$ can be written as
\begin{align}\label{eee R1}
R_{S_\lz,1}(f)(x_1,x_2)
= A_1(f)(x_1,x_2)+A_2(f)(x_1,x_2)+A_3(f)(x_1,x_2)+A_4(f)(x_1,x_2),
\end{align}
where
\begin{align*}
A_1(f)(x_1,x_2)&:=\int_0^{x_1\over2} R_{S_\lz,1}(x_1,y_1)f(y_1,x_2)dy_1\\
A_2(f)(x_1,x_2)&:=\bigg( p.v. \int_{x_1\over2}^{3x_1\over2} R_{S_\lz,1}(x_1,y_1)f(y_1,x_2)dy_1-{1\over \pi} \mathcal{T}_{\R_+,1}(f)(x_1,x_2)\bigg)\\
A_3(f)(x_1,x_2)&:=\int_{3x_1\over2}^\infty R_{S_\lz,1}(x_1,y_1)f(y_1,x_2)dy_1\\
A_4(f)(x_1,x_2)&:={1\over \pi} \mathcal{T}_{\R_+,1}(f)(x_1,x_2).
\end{align*}
Symmetrically, we can write
\begin{align}\label{eee R2}
R_{S_\lz,2}(f)(x_1,x_2)
&= B_1(f)(x_1,x_2)+B_2(f)(x_1,x_2)+B_3(f)(x_1,x_2)+B_4(f)(x_1,x_2),
\end{align}
where
\begin{align*}
B_1(f)(x_1,x_2)&:=\int_0^{x_2\over2} R_{S_\lz,2}(x_2,y_2)f(x_1,y_2)dy_2\\
B_2(f)(x_1,x_2)&:=\bigg( p.v. \int_{x_2\over2}^{3x_2\over2} R_{S_\lz,2}(x_2,y_2)f(x_1,y_2)dy_2-{1\over \pi} \mathcal{T}_{\R_+,2}(f)(x_1,x_2)\bigg)\\
B_3(f)(x_1,x_2)&:=\int_{3x_2\over2}^\infty R_{S_\lz,2}(x_2,y_2)f(x_1,y_2)dy_2\\
B_4(f)(x_1,x_2)&:={1\over \pi} \mathcal{T}_{\R_+,2}(f)(x_1,x_2).
\end{align*}

From the kernel upper bound $(i)'$ in Section \ref{subsec: riesz}, we obtain that
\begin{align}\label{eee A1B1 bound}
\|A_1(f)\|_{L^1(\rlz)}+ \|B_1(f)\|_{L^1(\rlz)}\leq C\|f\|_{L^1(\rlz)},
\end{align}
and similarly,
from the kernel upper bound  $(ii)'$ in Section \ref{subsec: riesz}, we obtain that
\begin{align}\label{eee A3B3 bound}
 \|A_3(f)\|_{L^1(\rlz)}+\|B_3(f)\|_{L^1(\rlz)}\leq C\|f\|_{L^1(\rlz)},
\end{align}

Next, from the kernel upper bound  $(iii)$ in Section \ref{subsec: riesz}, we obtain that
$$|A_2(f)(x_1,x_2)| \leq C \int_{x_1\over2}^{3x_1\over2} {1\over y_1} \bigg( 1+\log_+\Big(1+ {\sqrt{x_1y_1}\over |x_1-y_1|}\Big)\bigg) |f(y_1,x_2)| dy_1. $$
And from this, it is a direct calculation to verify
\begin{align}\label{eee A4 bound}
\|A_2(f)\|_{L^1(\rlz)}\leq C\|f\|_{L^1(\rlz)}.
\end{align}
Similarly, by a direction calculation
\begin{align}\label{eee B4 bound}
\|B_2(f)\|_{L^1(\rlz)}\leq C\|f\|_{L^1(\rlz)}.
\end{align}


We now show that the function $R_{S_\lz,1}(f)$ is in $L^1(\rlz)$. In fact,  from the equality \eqref{eee R1} and the estimates in \eqref{eee A1B1 bound}, \eqref{eee A3B3 bound} and \eqref{eee A4 bound}, we obtain  that
\begin{align} \label{eeee R1}
\|R_{S_\lz,1}(f) \|_{ L^1(\rlz)}&\leq \|A_4(f)\|_{L^1(\rlz)} +C\|f\|_{L^1(\rlz)}
\\&\leq C\|\mathcal{T}_{\R_+,1}(f)\|_{L^1(\rlz)} +C\|f\|_{L^1(\rlz)}.\nonumber
\end{align}
Similarly, we get that  that the function $R_{S_\lz,2}(f)$ is in $L^1(\rlz)$, which follows from the equality \eqref{eee R2} and the estimates in \eqref{eee A1B1 bound}, \eqref{eee A3B3 bound} and \eqref{eee B4 bound}. Moreover, we have
\begin{align} \label{eeee R2}
\|R_{S_\lz,2}(f) \|_{ L^1(\rlz)}&\leq \|B_4(f)\|_{L^1(\rlz)} +C\|f\|_{L^1(\rlz)}
\\&\leq C\|\mathcal{T}_{\R_+,2}(f)\|_{L^1(\rlz)} +C\|f\|_{L^1(\rlz)}.\nonumber
\end{align}


We now consider  $R_{S_\lz,1}R_{S_\lz,2}(f) $. From the equalities \eqref{eee R1} and \eqref{eee R2}, we have
we obtain that
\begin{align}\label{eee RR approx}
  &R_{S_\lz,1}R_{S_\lz,2}(f)(x_1,x_2) = \sum_{i=1}^4\sum_{j=1}^4 A_iB_j(f)(x_1,x_2).
\end{align}

From the kernel upper bounds $(i)'$ and $(ii)'$ in Section \ref{subsec: riesz}, and the estimates for $A_2$ and $B_2$ above, we obtain that
\begin{align}\label{eee AB bound}
 & \sum_{i=1}^3\sum_{j=1}^3 \|A_iB_j(f)\|_{L^1(\rlz)}\leq C\|f\|_{L^1(\rlz)},
 \end{align}

Based on the estimate in \eqref{eee AB bound}, we obtain that
$R_{S_\lz,1}R_{S_\lz,2}(f) \in L^1(\rlz)$ and we have
\begin{align}   \label{eeee RR}
&\|R_{S_\lz,1}R_{S_\lz,2}(f) \|_{ L^1(\rlz)}\\
&\leq C \sum_{i=1}^3 \|A_iB_4(f)\|_{L^1(\rlz)}+C\sum_{j=1}^3 \|A_4B_j(f)\|_{L^1(\rlz)}+C\|f\|_{L^1(\rlz)}\nonumber\\
&\leq  C\|\mathcal{T}_{\R_+,1}\mathcal{T}_{\R_+,2}f\|_{L^1(\rlz)}+C\|\mathcal{T}_{\R_+,2}f\|_{L^1(\rlz)}+C\|\mathcal{T}_{\R_+,1}f\|_{L^1(\rlz)}+C\|f\|_{L^1(\rlz)}.\nonumber
\end{align}

Combining the estimates in \eqref{eeee R1}, \eqref{eeee R2} and \eqref{eeee RR}, we obtain that
$f$ is in $H^1_{Riesz}(\rlz)$.


Conversely, suppose $f\in H^1_{Riesz}(\rlz)\cap L^2(\rlz)$. From the equalities \eqref{eee R1}, \eqref{eee R2} and \eqref{eee RR approx}, we obtain that
\begin{align*}
\|\mathcal{T}_{\R_+,1}(f)\|_{L^1(\rlz)}&\leq C\|R_{S_\lz,1}(f) \|_{ L^1(\rlz)} +C\|f\|_{L^1(\rlz)},\\
\|\mathcal{T}_{\R_+,2}(f)\|_{L^1(\rlz)}&\leq C\|R_{S_\lz,2}(f) \|_{ L^1(\rlz)} +C\|f\|_{L^1(\rlz)},\\
\|\mathcal{T}_{\R_+,1}\mathcal{T}_{\R_+,2}f\|_{L^1(\rlz)}
&\leq  C\|R_{S_\lz,1}R_{S_\lz,2}(f) \|_{ L^1(\rlz)}+C\|R_{S_\lz,1}(f) \|_{ L^1(\rlz)}\\
&\quad+C\|R_{S_\lz,2}(f) \|_{ L^1(\rlz)}+C\|f\|_{L^1(\rlz)},
\end{align*}
implying that $f\in H^1_{\mathcal{T}}(\rlz)$.
\end{proof}

\section{Applications: proofs of Theorems \ref{thm H1 com} and \ref{thm commutator}}

We first mention the definition of the classical product BMO space on $\R_+\times\R_+$.
We now consider $\R_+\times\R_+$ as a product spaces of homogeneous type, and then for the space ${\rm BMO}(\rlz)$,
we just refer to definition in product spaces of homogeneous type in \cite{HLL,HLL2}.
From \cite[Theorem 1.2]{HLL}, we obtain that
the dual of $H^1(\rlz)$ is ${\rm BMO}(\rlz)$.

We now provide the definition of product BMO
space associated with $S_\lz$.
\begin{defn}
Suppose $f\in L^1_{loc}(\rlz)$. We say that $f\in {\rm BMO}_{S_\lz}(\rlz)$ if
$$ \|f\|_{ {\rm BMO}_{S_\lz}(\rlz)}:= \sup_{\Omega} {1\over |\Omega|} \sum_{R \subset \Omega} S_R^2(f) <\infty.$$
Here the suprema is taken over all open sets $\Omega \subset \R_+\times\R_+$ with finite measures, the summation is taken
over all dyadic rectangles $R\subset\Omega$, and
$$ S_R^2(f) = \iint_{T(R)} |Q^{(1)}_{t_1}Q^{(2)}_{t_2}(f)(y_1,y_2)|^2 {dy_1dt_1dy_2dt_2\over t_1t_2} $$
with $Q^{(i)}_{t_i} := -t_i {d\over dt_i} \mathbb{P}_{t_i}^{[\lz]}$ for $i=1,2$.
\end{defn}
From \cite[Theorem 4.4]{DSTY}, we obtain that
the dual of $H^1_{S_\lz}(\rlz)$ is ${\rm BMO}_{S_\lz}(\rlz)$.

\begin{proof}[Proof of Theorem \ref{thm H1 com}]
Suppose $f\in H^1(\rlz)\cap L^2(\rlz)$. Then we have the atomic decomposition of $f$ (see \cite{HLLin}):
$$f=\sum\lambda_ja_j$$ such that
$\sum_{j=0}^\infty|\lambda_j|\leq 2\|f\|_{H^1(\rlz)}$, where the series converges in the sense of  $L^2(\rlz)$ and $H^1(\rlz)$, and
each $a_j$ is a product atom as follows.

A function $a(x_1, x_2)\in L^2(\rlz)$ is a
product atom     if it satisfies

\medskip

\noindent $ 1)$  {\rm supp} $a\subset \Omega$, where $\Omega$
is an open set of $\rlz$ with finite measure;

\medskip

\noindent $ 2)$  $\|a\|_{L^2(\rlz)}\leq
|\Omega|^{-{1\over 2}}$;

\medskip
\noindent $ 3)$ $a$ can be further decomposed into
$$
a=\sum\limits_{R\in m(\Omega)} a_R
$$

\noindent where $m(\Omega)$ is the set of all maximal dyadic
subrectangles of $\Omega$,    such that

 \smallskip

 (i) \ \ {\rm supp}\  $a_R\subset 10R$;
\smallskip

 (ii) $$\int_{\R_+} a_R(x_1, x_2)dx_1=\int_{\R_+} a_R(x_1, x_2)dx_2=0;$$

 \smallskip

 (iii) \
$$
\sum_{R\in m(\Omega)}
\big\|a_R\big\|_{L^2(\rlz)}^2\leq |\Omega|^{-1}.
$$

As a consequence, it is direct that there exists a positive constant $C$ such that for every product atom $a$,
$$ \|R_{S_\lambda,1}R_{S_\lambda,2}(a)\|_{ L^1(\rlz)} \leq C,
 \|R_{S_\lambda,1}(a)\|_{ L^1(\rlz)} \leq C, \,\,
\textrm{and}\,\,
\|R_{S_\lambda,2}(a)\|_{ L^1(\rlz)} \leq C, $$
all implying
$\|a\|_{H^1_{Riesz}(\rlz)} \leq C.  $
For the detail of the proof, we refer to \cite{HLLin}.  Thus, for $f\in H^1(\rlz)\cap L^2(\rlz)$, we have
\begin{eqnarray*}
\|f\|_{H^1_{Riesz}(\rlz)} \leq \sum_{j=0}^\infty|\lambda_j| \|a\|_{H^1_{Riesz}(\rlz)} \leq C \|f\|_{ H^1(\rlz)}.
\end{eqnarray*}
Since $ H^1(\rlz)\cap L^2(\rlz)$ is dense in $H^1(\rlz)$, we have that
for every $f\in H^1(\rlz)$,
$\|f\|_{H^1_{Riesz}(\rlz)}  \leq C \|f\|_{ H^1(\rlz)}.$
Thus, we get that the classical product Hardy space $H^1(\rlz)$ is a subspace of $H^{1}_{S_\lambda}( \rlz )$, i.e.
$ H^1(\rlz) \subset H^{1}_{S_\lambda}( \rlz ). $

Next, we point out that $H^1(\rlz)$ is a proper subspace of $H^{1}_{S_\lambda}( \rlz )$.
To see this, note that from Theorem \ref{thm main 2}, we obtain that
$H^{1}_{S_\lambda}( \rlz )$ coincides with $H^1_o(\rlz)$.
We now choose
$ f(x_1,x_2) = \chi_{Q_0}(x_1,x_2), $
where $Q_0=(0,1]\times(0,1]$ is the unit cube in $\R\times\R$. It is direct to see that the product odd extension $f_o$ is in
$H^1(\R\times\R)$, and hence this function $f$ is in $H^{1}_{S_\lambda}( \rlz )$. However, it is not in the product Hardy space
$H^1(\rlz)$ since it lacks cancellation.
Thus, we further have
$ H^1(\rlz) \subsetneq H^{1}_{S_\lambda}( \rlz ). $

As a consequence, we obtain that
 ${\rm BMO}_{S_\lambda}( \rlz ) $ is contained in the classical product BMO space ${\rm BMO}( \rlz )$, i.e.,
$ {\rm BMO}_{S_\lambda}( \rlz )  \subsetneq  {\rm BMO}( \rlz ). $
\end{proof}

We now provide the proof of Theorem \ref{thm commutator}.

\begin{proof}[Proof of Theorem \ref{thm commutator}]

From the kernel estimates of $(i)'$ and $(ii)'$ of the Riesz transform as in Section \ref{subsec: riesz}, we see that
$R_{S_\lz,1}$ and $R_{S_\lz,2}$ are standard Calder\'on--Zygmund operators. Hence, the composition
$R_{S_\lz,1}R_{S_\lz,2}$ are standard product Calder\'on--Zygmund operators.

Based on the general result of upper bound for the iterated commutator and product BMO space on space of homogeneous type (\cite[Theorem 3.3]{DLOWY}), we obtain that
\begin{align*}
\|[ [b, R_{S_\lambda,1}],R_{S_\lambda,2}] \|_{L^2( \rlz)\to L^2(\rlz)} \ls \|b\|_{{\rm BMO}_{S_\lambda} ( \rlz)}.
\end{align*}
In fact, for functions $b$ in the classical product BMO space ${\rm BMO} ( \rlz)$, we also have
\begin{align*}
\|[ [b, R_{S_\lambda,1}],R_{S_\lambda,2}] \|_{L^2( \rlz)\to L^2(\rlz)} \ls \|b\|_{{\rm BMO}( \rlz)}.
\end{align*}
From Theorem \ref{thm H1 com}, we know that
$  {\rm BMO}_{S_\lambda}( \rlz )  \subsetneq  {\rm BMO}( \rlz ). $
We now choose a particular function $b_0 \in {\rm BMO}( \rlz ) \backslash {\rm BMO}_{S_\lz}( \rlz )$, then we know that
the iterated commutator $[ [b_0, R_{S_\lambda,1}],R_{S_\lambda,2}]$ is bounded, which gives
$$  \infty= \|b\|_{{\rm BMO}_{S_\lambda} ( \rlz)}  \not\ls \|[ [b_0, R_{S_\lambda,1}],R_{S_\lambda,2}] \|_{L^2( \rlz)\to L^2(\rlz)}<\infty. $$
\end{proof}

\bigskip
\bigskip
{\bf Acknowledgments:} The authors would like to thank the referee for careful reading and helpful suggestions.
J. Betancor is supported by MTM2016-79436-P.  X. T. Duong and
J. Li are supported by ARC DP 160100153. J. Li is also supported by a Macquarie University New Staff Grant.
B. D. Wick's research supported in part by National Science Foundation
DMS grant \#1560955.
D. Yang is supported by the NNSF of China (Grant No. 11571289)
  and the NSF of Fujian Province of China (No. 2017J01011). Ji Li would like to thank Yumeng Ou for helpful discussions.

\end{document}